\documentclass[12pt,a4paper]{article}
\usepackage{amssymb,amsfonts,amsmath,amsbsy,amscd}
\usepackage{xcolor}
\usepackage{comandi}
\newcommand{\ol}[1]{\overline{#1}}
\numberwithin{equation}{section}

\newcommand{\R}{\ensuremath{\mathbb{R}}}

\newcommand{\Rn}{\ensuremath{\mathbb{R}^n}}

\newcommand{\N}{\ensuremath{\mathbb{N}}}

\newcommand{\E}{\ensuremath{\mathcal{E}}}

\newcommand{\loc}{\operatorname{loc}}
\newcommand{\pv}{\operatorname{p.v.}}

\newcommand{\sd}{\, \mathrm{d}}
\newcommand{\supp}{\operatorname{supp}}
\newcommand{\eps}{\ensuremath{\varepsilon}}
\newcommand{\weight}[1]{\langle #1\rangle}

\theoremstyle{plain}

\newtheorem{cor}[theorem]{Corollary}
\newtheorem{lem}[theorem]{Lemma}

\theoremstyle{definition}

\newtheorem{claim*}{Claim}
\newtheorem{assumption}[theorem]{Assumption}

\theoremstyle{remark}

\newcommand{\Hzero}{H^{-1}_{(0)}}
\newcommand{\Hone}{H^1_{(0)}}
\newcommand{\A}{\mathcal{A}}
\newcommand{\B}{\mathcal{B}}
\newcommand{\LL}{\mathcal{L}}
\newcommand{\dom}{\operatorname{dom}}

\newenvironment{proof*}[1]{{\bf Proof
#1:}}{\hspace*{\fill}\rule{1.2ex}{1.2ex}\\ }

\newcommand{\mean}[1]{m(#1)}

\begin{document}

\begin{titlepage}
\title{Cahn-Hilliard Equation with Nonlocal Singular Free Energies }
\author{  Helmut Abels\footnote{Fakult\"at f\"ur Mathematik,
    Universit\"at Regensburg,
    93040 Regensburg,
    Germany, {\sf helmut.abels@mathematik.uni-regensburg.de}
}\,,
Stefano Bosia\footnote{Politecnico di Milano, Dipartimento di Matematica,
    20133 Milano,
    Italy, {\sf stefano.bosia@polimi.it}
}\,,
Maurizio Grasselli\footnote{Politecnico di Milano, Dipartimento di Matematica,
    20133 Milano,
    Italy, {\sf mau\-ri\-zio.grasselli@polimi.it}
}}
\end{titlepage}

\maketitle

\begin{abstract}
We consider a Cahn-Hilliard equation which is the conserved gradient flow of a nonlocal total free energy functional. This functional is characterized by a Helmholtz free energy density, which can be of logarithmic type. Moreover, the spatial interactions between the different phases are modeled by a singular kernel. As a consequence, the chemical potential $\mu$ contains an integral operator acting on the concentration difference $c$, instead of the usual Laplace operator. We analyze the equation on a bounded domain subject to no-flux boundary condition for $\mu$ and by assuming constant mobility. We first establish the existence and uniqueness of a weak solution and some regularity properties. These results allow us to define a dissipative dynamical system on a suitable phase-space and we prove that such a system has a (connected) global attractor. Finally, we show that a Neumann-like boundary condition can be recovered for $c$, provided that it is supposed to be regular enough.
\end{abstract}

\noindent{\bf Key words}: Cahn-Hilliard equation, nonlocal free energy, regional fractional Laplacian, logarithmic potential, monotone operators, global attractors.

\medskip
\noindent{\bf MSC2010:} 35B41, 37L99, 45K05, 47H05, 47J35, 80A22

\section{Introduction}

The Cahn-Hilliard equation was proposed long ago as a diffuse interface model for phase separation in binary alloys subject to a cooling process (see \cite{Cahn-61,CahnHilliard}, cf. also \cite{Fife-2000,N-C-98,N-C-08} and references therein). Since then it has been studied theoretically by many authors (see the pioneering contributions \cite{Debussche1995,EG-96,EZ-86,NST-89}, cf. also the review paper \cite{CMZ-11}). More recently, it has been observed that a physically more rigorous derivation leads to a nonlocal equation (see \cite{GL-97,GL-98}). In this case, one can still view the Cahn-Hilliard equation as a conserved gradient flow of the first variation of a suitable total free energy functional $E$. However, $E$ has the following form
\begin{equation*}
    E(c)= \frac12\int_\Omega\int_\Omega (c(x)-c(y))^2k(x,y,x-y)\sd x\sd y +\int_\Omega f(c(x)) \sd x.
\end{equation*}
Here $c$ denotes the (relative) concentration difference of the two components, $\Omega\subset\Rn$ is a bounded domain with a $\Cont[2]$-boundary $\partial\Omega$  (where $n\in\{1,2,3\}$ in applications), $k$ is the interaction kernel and $f$ is the Helmholtz free energy density. The latter, which accounts for the entropy of the system, is given by (see~\cite{CahnHilliard})
\begin{equation}\label{logpot}
    f(c) = \frac{\theta}2 \left((1+c)\ln (1+c)+ (1-c)\ln (1-c)\right) - \frac{\theta_c}2 c^2,
\end{equation}
where $\theta,\theta_c>0$ are given constants. We remind that $f$ is usually approximated in the literature by a more tractable fourth-order polynomial double well. Moreover, we note that $f$ is convex if and only if $\theta\geqslant \theta_c$. In this case the mixed phase is stable. On the other hand, if $0<\theta<\theta_c$, the mixed phase is unstable and phase separation occurs.

The chemical potential $\mu$ is the first variation of $E$ and the Cahn-Hilliard equation can be then written as follows
\begin{equation}\label{eq:CH}
\partial_t c = \nabla\cdot(m(c)\nabla \mu),
\end{equation}
where $m(c)$ is the so-called mobility. Here we assume for simplicity that $m(c)$ is independent of $c$ and, for simplicity, equal to one. When $k(x,y,x-y)=J(x-y)$ and $J: \Rn \to \R $ is a sufficiently smooth (and even) function, equation~\eqref{eq:CH} has been analyzed in~\cite{GZ-03,GG-14,GL-98,LP-DCDS-S-11,LP-JMAA-11} (see also~\cite{BH-D05,BH-N05,CKRS-07} for constant mobility and/or regular $f$). But in this case the mathematical properties of solutions are very different from the usual so-called local Cahn-Hilliard equation. More precisely a second order differential operator in the equation for the chemical potential is replaced by a compact operator  in this case. As far as we know, such an equation has never been studied when $k$ is singular (see, however, \cite{NNG-08} for a fractional Allen-Cahn equation). This is the goal of the present contribution. Therefore we consider the following problem
\begin{alignat}{2}
    \partial_t c &= \Delta \mu\ &\qquad&\text{in}\ \Omega\times (0,\infty), \label{eq:CH1}\\
    \mu &= \LL c + f'(c)&\qquad&\text{in}\ \Omega\times (0,\infty), \label{eq:CH2}\\
    \dnu{\mu} &=0&\qquad&\text{on}\ \partial\Omega\times (0,\infty), \label{eq:CH3}\\
    c|_{t=0}&= c_0 &&\text{in}\ \Omega, \label{eq:CH4}
\end{alignat}
where $\LL$ is a non-local linear operator defined as follows
\begin{align}\label{eq:defnL}
    \LL u(x) &= \pv \int_{\Omega} (u(x)-u(y))k(x,y,x-y)dy\\\nonumber
    &=\lim_{\eps\to 0} \int_{\Omega\setminus B_\eps(x)} (u(x)-u(y))k(x,y,x-y)dy
\end{align}
and $k\colon \Rn\times \Rn\times (\Rn\setminus\{0\})\to \R$ is $(n+2)$-times continuously differentiable and satisfies the following conditions (see~\cite{Abels2007}):
\begin{alignat}{2}
    &\quad  k(x,y,z)=k(y,x,-z)\,, \label{k-ass-one} \\
    &\quad |\partial_x^\beta\partial_y^\gamma\partial_z^\delta k(x,y,z)| \leqslant
        C_{\beta,\gamma,\delta}|z|^{-n-\alpha-|\delta|} \, , \label{k-ass-two} \\
    &\quad  c_0 |z|^{-n-\alpha} \leqslant k(x,y,z)\leqslant C_0 |z|^{-n-\alpha} \,. \label{k-ass-three}
\end{alignat}
for all $x,y,z\in\Rn$, $z\neq 0$ and $\beta, \gamma, \delta\in\N_0^n$ with $|\beta|+|\gamma|+|\delta|\leqslant n+2$ where $\alpha$ is the order of the operator. Unless specified otherwise, throughout this paper we will always consider the case $\alpha \in (1,2)$. An example for $k(\cdot, \cdot, \cdot)$ is given by $k(x,y,z) = \omega(x,y) |z|^{-n-\alpha}$ and $\omega\in \Cont[n+2]_{b}(\Rn)$. Note that the definition of the operator $\LL$ depends on $\Omega$. Formally, in the case $\Omega = \Rn$ and $k(x,y,z) = |z|^{-n-\alpha}$ one has $\LL = const \times (-\Delta)^\frac{\alpha}{2}$ where $(-\Delta)^\frac{\alpha}{2}$ is a fractional power of the Laplace operator. If $\Omega$ is a bounded domain, the operator $\LL$ has the same form as the generator of a censored stable process (cf., e.g., \cite{BBC03}) and it is also known as regional fractional Laplacian.

Our main result is the well-posedness of the weak formulation of problem~\eqref{eq:CH1}-\eqref{eq:CH4} together with a natural boundary condition for $c$, which will be part of the weak formulation. In the above (strong) formulation~\eqref{eq:CH1}--\eqref{eq:CH4}, a boundary condition for the variable $c$ is missing. A further result is concerned  with the characterization of such a condition, provided that the weak solution is smooth enough (say,  $c\in \Cont[1, \beta](\closure{\Omega})$) and $k$ fulfills suitable assumptions. More precisely, we prove that
\begin{equation*}
    \nabla c(x_{0}) \cdot \vect{n}_{x_{0}}  = 0,
\end{equation*}
where $\vect{n}_{x_{0}}$ depends on the interaction kernel $k$ (see \eqref{eq:direction_kernel} below).

This condition reduces to the usual homogeneous Neumann boundary condition for $c$ for symmetric kernels (cf.\ Theorem~\ref{thm:boundary_regularity} and Remarks~\ref{rem:HomKernel} and \ref{rem:IsoKernel} below). Unfortunately, we are unable to prove that a weak solution is indeed as regular as it is required for this characterization. Nonetheless, our weaker regularity results allow us to prove that the dissipative dynamical
system generated by \eqref{eq:CH1}-\eqref{eq:CH4} has a (connected) global attractor.

The paper is organized as follows. In Section~\ref{prelim} we introduce some basic notation and function spaces as well as we account for some preliminary results. Section~\ref{subgra} is essentially devoted to the computation of the subgradient of the (convex) functional
\begin{equation*}
    F(c) =  \frac12\int_\Omega\int_\Omega (c(x)-c(y))^2k(x,y,x-y)\sd x\sd y + \int_{\Omega} \phi( c(x)) \mathrm{d}x
\end{equation*}
and to the characterization of its domain. Here $\phi$ is the convex part of $f$ (see Assumption \ref{assump:1} below). Combining the results of Sections~\ref{prelim} and~\ref{subgra} we give the proof of the well-posedness theorem in Section~\ref{S:existence}. The existence of the global attractor is established in Section~\ref{S:attractor}. Finally, in Section~\ref{S:BC_regular} we show that a regular weak solution $c$ does satisfy the above boundary condition.

\section{Basic tools and well-posedness}\label{prelim}

Given a set $M$, its power set will be denoted by $\mathcal{P}(M)$. Moreover, we denote $\R^n_+=\{x\in \R^n: x_n >0\}$ and $\R_+=\R^1_+$. If $X$ is a (real) Banach space and $X^\ast$ is its dual, then
\begin{equation*}
    \weight{f,g}\equiv\weight{f,g}_{X^\ast,X} = f(g), \qquad f\in X^\ast,g\in X,
\end{equation*}
denotes the duality product. Moreover, if $H$ is a (real) Hilbert space, $(\cdot,\cdot )_H$ will indicate its inner product. In the following, all Hilbert spaces will be separable.

\subsection{Function spaces}\label{eq:FctSpaces}\label{S:funspa}

Throughout the paper $\Omega\subseteq\R^n$ will be a bounded domain with $\Cont[2]$-boundary. Let $L^p(\Omega)$, $1\leqslant p\leqslant \infty$, be the set of $p$-integrable (or essentially bounded) functions $f\colon \Omega\to \R$ and set $\|\cdot \|_p\equiv \|\cdot\|_{L^p(\Omega)}$. Moreover, $H^m(\Omega)$, $m\in \N$, indicates the usual $L^2$-Sobolev space of order $m$ and $H^m_0(\Omega)$ is the closure of $\Cont[\infty]_{0}(\Omega)$ in $\Hs[m](\Omega)$.

Given $f\in L^1(\Omega)$, we set
\begin{equation*}
    \mean{f} \eqdef \frac{1}{|\Omega|}\int_\Omega f(x) \sd x
\end{equation*}
and, for $m \in \mathbb{R}$ we define
\begin{equation*}
    L^2_{(m)}(\Omega) := \{ f \in L^2(\Omega) \mid \mean{f} = m \},
\end{equation*}
so that $P_0 f \eqdef f-m(f)$ denotes the orthogonal projection onto $L^2_{(0)}(\Omega)$.

We then introduce
\begin{equation*}
    \Hone=\Hone (\Omega)= \left\{c\in H^1(\Omega) \mid \mean{c} = 0 \right\}
\end{equation*}
equipped with the inner product
\begin{equation*}
    (c,d)_{\Hone(\Omega)} = (\nabla c,\nabla d)_{L^2(\Omega)},\qquad c,d\in \Hone(\Omega).
\end{equation*}
Observe that $\Hone(\Omega)$ is a Hilbert space due to Poincar\'e's inequality. Moreover, let $\Hzero\equiv\Hzero(\Omega)= \Hone(\Omega)^\ast$ and consider the Riesz isomorphism $\mathcal{R}\colon \Hone(\Omega)\to \Hzero(\Omega)$ given by
\begin{equation*}
    \weight{\mathcal{R} c,d}_{\Hzero,\Hone} = (c,d)_{\Hone}= (\nabla c,\nabla d)_{L^2}, \qquad c,d\in \Hone(\Omega),
\end{equation*}
i.e., $\mathcal{R}= -\Delta_N$ is the Laplacian with Neumann boundary conditions in the variational sense. Therefore we equip $\Hzero(\Omega)$ with the inner product
\begin{equation*}
    (f,g)_{\Hzero} = (\nabla \Delta_N^{-1} f, \nabla \Delta_N^{-1} g)_{L^2} =(\Delta_N^{-1} f, \Delta_N^{-1} g)_{\Hone}.
\end{equation*}

Moreover, we embed $\Hone(\Omega)$ and $L^2_{(0)}(\Omega)$ into $\Hzero(\Omega)$ in the canonical way, that is,
\begin{equation*}
    \weight{c,\varphi}_{\Hzero,\Hone} = \int_\Omega c(x) \varphi(x) \sd x,\qquad \forall \, \varphi \in \Hone(\Omega), c \in \Lp_{(0)}(\Omega).
\end{equation*}

Finally, we need to introduce the so-called fractional \Lp-Sobolev-Slobodeckii spaces as follows. Let $s \in (0,1)$. Then,
for any $u \in \Lp(\Omega)$, set
\begin{equation*}
    \norm{\Hs[s](\Omega)}{u}^{2} \eqdef \Lpnorm{u}^{2} + \int_{\Omega} \int_{\Omega} \frac{|u(x) - u(y)|^{2}}{|x-y|^{n + 2s}} \, \mathrm{d}x \mathrm{d}y
\end{equation*}

and
\begin{equation*}
    \Hs[s](\Omega) \eqdef \{ f \in \Lp(\Omega) \mid \norm{\Hs[s](\Omega)}{f} < \infty \}.
\end{equation*}
Let us denote by $\Hs[s]_{0}(\Omega)$ the closure of $\Cont[\infty]_{0}(\Omega)$ in $\Hs[s](\Omega)$, while $\Hs[-s](\Omega)$ and $\Hs[-s]_{(0)}(\Omega)$ will be the dual spaces of $\Hs[s](\Omega)$ and $\Hs[s]_{(0)}(\Omega)$, respectively. We refer the reader to~\cite{Adams} for the interpolation results for such spaces which will be used hereafter.

\subsection{Weak formulation and main result}\label{eq:exisuniq}

Before introducing a weak formulation of our problem we state our assumptions on $f$ which are satisfied by the physically relevant case \eqref{logpot}.
Namely, we suppose
\begin{assumption}\label{assump:1}
    $f\colon [a,b]\to \R$, $a<0<b$, is a continuous function, which is twice continuously differentiable in $(a,b)$, such that
    \begin{equation*}
        \lim_{s\to a} f'(s)=-\infty, \qquad \lim_{s\to b} f'(s)=\infty,
    \end{equation*}
    and $f''(s)\geqslant -d$ for some $d\geqslant 0$.
\end{assumption}

Since $f$ is defined on an interval $[a,b]$, we also extend $f(x)$ by $+\infty$ if $x\notin[a,b]$. Hence $E(c)<\infty$ implies $c(x)\in [a,b]$ for almost every $x\in\Omega$. Note that, although $f$ is in general non-convex, it can be considered as a perturbation of a convex potential. Indeed, thanks to Assumption \ref{assump:1}, we have that there exists a positive number $d > 0$ and a continuous, convex and twice continuously differentiable in $(a,b)$ function $\phi \colon [a,b] \to \mathbb{R}$ such that the potential $f$ can be decomposed as $f(s) = \phi(s) - \frac{d}{2} s^{2}$. This will be the key point in the following analysis, which is based on a decomposition of the associated operators in a monotone operator plus a Lipschitz perturbation. The condition $\lim_{c \to a} \phi'(c) = -\infty$, $\lim_{c \to b} \phi'(c) = \infty$ will force $c$ to take values in the interval $[a,b]$ and ensures that the subgradient of the associated functional is single-valued with a suitable domain.

Let us introduce the symmetric bilinear form associated to $\LL$
\begin{equation*}
    \E(u,v)= \frac12\int_\Omega\int_\Omega (u(x)-u(y))(v(x)-v(y))k(x,y,x-y)\sd x\sd y
\end{equation*}
for all $u,v\in \Hs[\sfrac{\alpha}{2}](\Omega)$.

The notion of weak solution to problem~\eqref{eq:CH1}-\eqref{eq:CH2} is given by

\begin{definition}\label{D:variational_solutions}
    Let $c_{0}\in \Hs[\sfrac{\alpha}{2}](\Omega)$ such that $E(c_0) < \infty$ be given. A pair $(c, \mu)$ is a global (weak) solution to~\eqref{eq:CH1}-\eqref{eq:CH4} if $\mu \in \Bochner{\Lp}{0}{T}{\Hs(\Omega)}$ for all $T > 0$, $c \in \Bochner{\Lp[\infty]}{0}{\infty}{\Hs[\sfrac{\alpha}{2}]_{(0)}(\Omega)}$ and $\dt{c} \in \Bochner{\Lp}{0}{\infty}{\Hs[-1]_{(0)}(\Omega)}$ hold, if $(c, \mu)$ satisfies
    \begin{align}
        \longduality{\Hs[-1]_{(0)}}{\dt{c}(t)}{\eta}{\Hs_{(0)}}   &= -\Ltwoprod{\nabla \mu(t)}{\nabla \eta} \label{eq:PDE_variational1}\\
        (\mu(t),\varphi)_{L^2}   &= \mathcal{E}(c(t), \varphi) + (f'(c(t)),\varphi)_{L^2} \label{eq:PDE_variational2}
    \end{align}
    for all $\eta \in \Hs_{(0)}(\Omega)$, all $\varphi \in \Hs[\sfrac{\alpha}{2}](\Omega)$ and a.e.\ $t > 0$, and if
    \begin{equation*}
        \lim_{t \to 0} c(t) = c_{0} \qquad \text{in $\Hs[\sfrac{\alpha}{2}](\Omega)$}.
    \end{equation*}
\end{definition}

The main result of this paper is the following.
\begin{theorem}\label{thm:Existence}
    Let Assumption~\ref{assump:1} hold. For every $c_0\in H^{\alpha/2}(\Omega)$ with $E(c_0)<\infty$, there is a unique (global) solution $c\in BC([0,\infty);H^{\alpha/2}(\Omega))$ to~\eqref{eq:CH1}-\eqref{eq:CH4} in the sense of Definition~\ref{D:variational_solutions} which satisfes the energy identity
    \begin{equation}\label{eq:energy_identity}
        E(c(T))+ \int_0^T \|\nabla \mu(t)\|_{L^2(\Omega)}^2 \sd t = E(c_0)
    \end{equation}
    for all $T>0$. Furthermore, the following regularity properties hold
    \begin{alignat*}{1}
        \kappa\phi'(c)&\in \Bochner{\Lp[\infty]}{0}{\infty}{\Lp(\Omega)}, \\
        \kappa\mu&\in \Bochner{\Lp[\infty]}{0}{\infty}{\Hs(\Omega)},\quad \text{and}\\
        \kappa \partial_t c&\in \Bochner{\Lp[\infty]}{0}{\infty}{\Hzero(\Omega)} \cap \Bochner{\Lp}{0}{\infty}{\Hs[\sfrac{\alpha}{2}]_{(0)}(\Omega)},
    \end{alignat*}
    where $\kappa(t)= \left(\frac{t}{1+t}\right)^\frac12$. In addition, if $n \leqslant 3$, then there is some $\beta>0$ depending only on $n$,
    such that
    \begin{equation*}
        \kappa c \in \Bochner{\Lp[\infty]}{0}{\infty}{\Cont[\beta](\closure{\Omega})}.
    \end{equation*}
    Finally, setting $Z_{m} \eqdef \left\{ \widetilde{c} \in \Hs[\sfrac{\alpha}{2}](\Omega) \mid E(\widetilde{c}) < \infty, \mean{\widetilde{c}} = m \right\}$, where $m\in (a,b)$ is given, the mapping $Z_m \ni c_0 \mapsto c(t) \in \Hs[\gamma]_{(m)}(\Omega)$, $\gamma < \tfrac{\alpha}{2}$ is strongly continuous.
\end{theorem}

\subsection{Evolution equations with monotone operators}\label{LipPer}

We refer, e.g., to Br\'ezis~\cite{Brezis} and Showalter~\cite{Showalter} for results in the theory of monotone operators. In the following we just summarize some basic facts and definitions. Let $H$ be a real-valued and separable Hilbert space. Recall that $\A\colon H\to \mathcal{P}(H)$ is a monotone operator if
\begin{equation*}
    (w-z,x-y)_H \geqslant 0 \qquad \text{for all}\ w\in \A(x), z\in \A (y).
\end{equation*}
Moreover, $\mathcal{D}(A)=\{x\in H: \A(x)\neq \emptyset\}$.
Now let $\varphi\colon H\to \R\cup \{+\infty\}$ be a convex function. Then $\dom (\varphi)= \{x\in H:\varphi (x)<\infty\}$ and $\varphi$ is called proper if $\dom (\varphi)\neq \emptyset$. Moreover, the subgradient $\partial\varphi\colon H\to \mathcal{P}(H)$ is defined by $w\in \partial \varphi(x)$ if and only if
\begin{equation*}
    \varphi(\xi)\geqslant \varphi(x)+ (w,\xi-x)_H \qquad \text{for all}\ \xi \in H.
\end{equation*}
Then $\partial\varphi$ is a monotone operator and, if additionally $\varphi$ is lower semicontinuous, then $\partial\varphi$ is maximal monotone, cf. \cite[Exemple 2.3.4]{Brezis}.

The proof of Theorem~\ref{thm:Existence} is based on the following result for the evolution problem associated to Lipschitz perturbations of monotone operators (see, e.g., \cite[Theorem~3.1]{AsymptoticCH})
\begin{theorem}\label{thm:MonotoneLipschitz}
    Let $H_0, H_1$ be real, separable Hilbert spaces such that $H_1$ is densely embedded into $H_0$. Moreover, let $\varphi\colon H_0 \to \R\cup \{+\infty\}$ be a proper, convex and lower semicontinuous functional such that $\varphi=\varphi_1+\varphi_2$, where $\varphi_2\geqslant 0$ is convex and lower semicontinuous, $\dom\varphi_1 = H_1$, and $\varphi_1|_{H_1}$ is a bounded, coercive, quadratic form on $H_1$ and set $\A= \partial\varphi$. Furthermore, assume that $\B\colon H_1\to H_0$ is a globally Lipschitz continuous function. Then for every $u_0\in \mathcal{D}(\A)$ and $f\in L^2(0,T;H_0)$ there is a unique $u\in W^1_2(0,T;H_0)\cap L^\infty(0,T;H_1)$ with $u(t)\in \mathcal{D}(\A)$ for a.e. $t>0$ solving
    \begin{eqnarray}
        \frac{du}{dt} (t) +  \A(u(t)) &\ni&  \B(u(t)) + f(t) \quad \textrm{for a.e.}\ t\in (0,T) \label{eq:1}\\
        u(0) &=& u_0
    \end{eqnarray}
    Moreover, $\varphi(u)\in L^\infty(0,T)$.
\end{theorem}

\subsection{Results on the nonlocal operator $\mathcal{L}$}\label{S:nonlocal_operator}

Assumptions~\eqref{k-ass-one}--\eqref{k-ass-three} allow us to deduce the following norm equivalence results.
\begin{lemma}
    Let $u \in \Hs[\sfrac{\alpha}{2}](\Omega)$. Then there exist two positive constants $c$ and $C$ such that
    \begin{equation*}
         c\norm{\Hs[\sfrac{\alpha}{2}](\Omega)}{u}^{2} \leqslant |\mean{u}|^2 + \E(u,u) \leqslant  C \norm{\Hs[\sfrac{\alpha}{2}](\Omega)}{u}^{2} \qquad \forall\,u \in \Hs[\sfrac{\alpha}{2}](\Omega).
    \end{equation*}
\end{lemma}

\begin{corollary}
    The following norm equivalences hold:
    \begin{align}
        &\mathcal{E}(u,u)\sim \norm{\Hs[\sfrac{\alpha}{2}]_{(0)}(\Omega)}{u}^2 \qquad \forall\, \Hs[\sfrac{\alpha}{2}]_{(0)}(\Omega),\\\label{eq:EquivNorm2}
        &\mathcal{E}(u,u) + | \mean{u} |^{2} \sim \norm{\Hs[\sfrac{\alpha}{2}](\Omega)}{u}^2 \qquad \forall\, \Hs[\sfrac{\alpha}{2}](\Omega).
    \end{align}
\end{corollary}

We now consider the variational extension of the nonlocal linear operator $\mathcal{L}$ (see~\eqref{eq:defnL}). More precisely, abusing the notation, we define $\mathcal{L} \colon \Hs[\sfrac{\alpha}{2}](\Omega) \to \Hs[-\sfrac{\alpha}{2}](\Omega)$ by setting
\begin{equation*}
    \longduality{\Hs[-\sfrac{\alpha}{2}]}{\mathcal{L}u}{\varphi}{\Hs[\sfrac{\alpha}{2}]} = \mathcal{E}(u, \varphi) \quad \text{for all $\varphi \in \Hs[\sfrac{\alpha}{2}](\Omega)$}.
\end{equation*}
In particular we have
\begin{equation*}
    \duality{\mathcal{L}u}{1} = \mathcal{E}(u, 1) = 0
\end{equation*}
by definition.

\begin{remark}
    This definition of $\mathcal{L}$ agrees with~\eqref{eq:defnL} as soon as $u \in \Hs[\alpha]_{loc}(\Omega)\cap \Hs[\sfrac{\alpha}2](\Omega)$ and $\varphi\in C_0^\infty(\Omega)$, cf. \cite[Lemma~4.2]{Abels2007}.
\end{remark}

We will also need the following regularity result, which essentially states that the operator $\mathcal{L}$ is of lower order with respect to the usual Laplace operator.
\begin{lemma}\label{L:regularity_regularization}
    Let $g \in \Lp_{(0)}(\Omega)$ and $\theta>0$. Then the unique solution $u$ the problem
    \begin{equation}\label{eq:regularized_nonlocal}
        - \theta \int_{\Omega} \nabla u \cdot \nabla \varphi + \mathcal{E}(u, \varphi) = \Ltwoprod{g}{\varphi} \qquad \text{in $\Omega$,}
    \end{equation}
    for all $\varphi \in \Hs_{(0)}(\Omega)$, belongs to $\Hs[2]_{\loc}(\Omega) \cap \Hs_{(0)}(\Omega)$ and satisfies the estimate
    \begin{equation*}
        \theta \|\nabla u\|^2_{L^2} + \|u\|_{H^{\alpha/2}}^2 \leqslant C\|g\|_{L^2}^2,
    \end{equation*}
    where $C$ is independent of $\theta>0$.
\end{lemma}

\begin{proof}
    Existence and uniqueness of a solution $u\in\Hs_{(0)}(\Omega)$to~\eqref{eq:regularized_nonlocal}  easily follow from the continuity and coercivity of the bilinear form $\mathcal{E}(\cdot, \cdot)$ through the Lax-Milgram theorem. Also, the estimate can be obtained by choosing $\varphi = u$. The claimed inner regularity $u \in \Hs[2]_{\loc}(\Omega)$ can be shown by arguing as in~\cite[Lemma~4.3]{Abels2007}.
\end{proof}

The following regularity result is more involved. Its proof is obtained by using ideas of the proof of~\cite[Lemma~5.4]{Abels2007}.

\begin{lemma}\label{L:continuity}
    Let  $\partial\Omega$ of class $\Cont[2]$ and let $u\in H^{\frac{\alpha}2}(\Omega)$ such that $\phi'(u)\in L^2(\Omega)$ and
    \begin{equation}\label{eq:Weak}
        \mathcal{E}(u,\varphi) +\int_\Omega \phi'(u)\varphi \, dx  =\int_\Omega g\varphi \, dx \qquad \forall\, \varphi\in H^{\frac{\alpha}2}(\Omega)
    \end{equation}
    for some given $g\in H^1(\Omega)$. Then $u\in \Cont[\beta](\overline\Omega)$ for some $\beta\in (0,1)$ depending only on $n$ and there is a constant $C>0$ independent of $u$ and $g$ such that
    \begin{equation}\label{eq:HoelderEstim}
        \|u\|_{\Cont[\beta](\overline\Omega)} \leqslant C\left(\|g\|_{H^1(\Omega)} + \|u\|_{H^{\alpha/2}(\Omega)}+\|\phi'(u)\|_{L^2(\Omega)}\right).
    \end{equation}
\end{lemma}

\begin{proof}
    Let us consider first the case of a half-space $\Omega=\R^n_+$. We will prove that $u \in H^{\frac{\alpha}2}(\R_+; H^1(\R^{n-1}))\cap L^2(\R_+;H^{1+\alpha/2}(\R^{n-1}))$ by approximating the tangential derivatives by difference quotients. Then, using the interpolation inequality
    \begin{equation*}
        \|f\|_{H^{1+s}(\R^{n-1})} \leqslant C \|f\|_{H^{1+\frac{\alpha}2}(\R^{n-1})}^{\frac{2s}\alpha}\|f\|_{H^1(\R^{n-1})}^{1-\frac{2s}\alpha}
    \end{equation*}
    and direct estimates, one obtains
    \begin{alignat*}{1}
        u &\in H^{\frac{\alpha}2}(\R_+; H^1(\R^{n-1}))\cap L^2(\R_+; H^{1+\frac{\alpha}2}(\R^{n-1}))\\
        &\hookrightarrow H^{\frac{\alpha}2-s}(\R_+; H^{1+s'}(\R^{n-1})) \hookrightarrow \Cont[\beta](\R_+, H^{1+s'}(\R^{n-1}))\hookrightarrow \Cont[\beta](\overline{\R^n_+}),
    \end{alignat*}
    for any $0<s'<s<\frac{\alpha}2-\frac12$, where we have used~\cite[Corollary 26]{VektorBesovSpaces}.

    We denote
    \begin{equation*}
        \tau_{j,s}f(x)= f(x+se_j), \quad\Delta_{j,h}^+ f(x)= \tau_{j,h}f(x) -f(x), \quad \Delta_{j,h}^- f(x)= f(x)-\tau_{j,-h}f(x),
    \end{equation*}
    for $h>0$,  where $e_j$ is the $j$-th canonical unit vector, $j=1,\ldots, n-1$. Replacing $\varphi$ by $-h^{-s}\Delta_{j,h}^- \varphi$ with $s\in [0,1]$, $j\in \{1,\ldots,n-1\}$ in (\ref{eq:Weak}), we obtain that $v_h=h^{-s} \Delta_{j,h}^+ u$ solves
    \begin{eqnarray*}
        \mathcal{E}(v_h, \varphi)+ h^{-s} \int_\Omega \Delta^+_{j,h}(\phi'(u))\varphi \, dx &=& -\mathcal{E}_{j,h} (\tau_{j,h} u,\varphi) -\int_\Omega g h^{-s}\Delta_{j,h}^-\varphi\, dx
    \end{eqnarray*}
    for all $\varphi\in \Cont[\infty]_{0}(\Omega)$, where $\mathcal{E}_{j,h}$ is the bilinear form with kernel $h^{-s}( k(x+he_j,y+he_j,z)-k(x,y,z))$. Note that by (\ref{k-ass-two}) the latter kernel is bounded by $C|z|^{-d-\alpha}$ uniformly in $h>0$.

    First we discuss an auxiliary estimate, which will be needed to deal with some terms in the localization procedure. To this end let $s\in (\frac12,\frac{\alpha}2)$. Then choosing $\varphi=v_h$, using (\ref{eq:EquivNorm2}) and
    \begin{eqnarray*}
        \lefteqn{\int_\Omega \Delta_{j,h}^+ (\phi'(u))\Delta_{j,h}^+ u\, dx  =}\\
        && \int_\Omega \left(\phi'(u(x+he_j))-\phi'(u(x))\right)(u(x+he_j)-u(x))\, dx \geqslant 0
    \end{eqnarray*}
    we conclude
    \begin{equation*}
        \|v_h\|_{H^{\frac{\alpha}2}(\R^n_+)}^2 \leqslant C \left(\left\|g\right\|_{L^2(\R^n_+)} \|h^{-s}\Delta_{j,h}^-v_h\|_{L^2(\R^n_+)}+\|v_h\|_{L^2(\R^n_+)}^2+ \|u\|_{H^{\alpha/2}(\R^n_+)}^2\right)
    \end{equation*}
    We now use the inequality
    \begin{equation*}
        \|h^{-s}\Delta_{j,h}^\pm w\|_{L^2(\R^n_+)}\leqslant C \|w\|_{H^s(\R^n_+)}\leqslant C \|w\|_{H^{\frac{\alpha}2}(\R^n_+)},
    \end{equation*}
    which  follows from interpolation of $\|h^{-1}\Delta_{j,h} w\|_{L^2(\R^n_+)}\leqslant C\|w\|_{H^1(\R^n_+)}$ and $\|\Delta_{j,h} w\|_{L^2(\R^n_+)}\leqslant 2\|w\|_{L^2(\R^n_+)}$.
    Hence, we have
    \begin{equation*}
        \sup_{j=1,\ldots,n-1} \left\|h^{-2s} (\Delta_{j,h}^+)^2 u \right\|_{L^2(\R^n_+)} \leqslant \|v_h\|_{H^{\frac{\alpha}2}(\R^n_+)}\leqslant C \left(\left\|g\right\|_{L^2(\R^n_+)} +\|u\|_{H^{\alpha/2}(\R^n_+)} \right),
    \end{equation*}
    which implies that $u\in L^2(\R_+;B^{2s}_{2,\infty}(\R^{n-1}))\hookrightarrow L^2(\R_+;H^1(\R^{n-1}))$ (cf. \cite[Theorem 6.2.5]{Interpolation}). Also, we get
    \begin{equation*}
        \sup_{j=1,\ldots,n-1} \left\|\partial_{x_j} u \right\|_{L^2(\R^n_+)} \leqslant C \left(\left\|g\right\|_{L^2(\R^n_+)} +\|u\|_{H^{\alpha/2}(\R^n_+)} \right)
    \end{equation*}

    Next we choose $s=1$ in the definition of $v_h$ and we obtain similarly
    \begin{eqnarray*}
        \|v_h\|_{H^{\frac{\alpha}2}(\R^n_+)} &\leqslant& C\left(\sup_{j=1,\ldots, n-1} \left\|\partial_{x_j}g\right\|_{L^2(\R^n_+)} + \|v_h\|_{L^2(\R^n_+)} + \|u\|_{H^{\alpha/2}(\R^n_+)}\right)\\
        &\leqslant & C\left(\sup_{j=1,\ldots, n-1} \left\|(\partial_{x_j}g,\partial_{x_j} u)\right\|_{L^2(\R^n_+)} + \|u\|_{H^{\alpha/2}(\R^n_+)}\right).
    \end{eqnarray*}
    Hence $v_h= h^{-1}\Delta_{j,h}^+ u$, $h>0$, is uniformly bounded in $H^{\frac{\alpha}2}(\R^n_+)$ and therefore $\partial_{x_j} u\in H^{\frac{\alpha}2}(\R^n_+)$.

    In order to prove the statement for a bounded domain $\Omega$, it is sufficient to show that for every $x\in \overline{\Omega}$  and for some open neighborhood $U$ of $x$ we have that $u\in \Cont[\beta](\overline{\Omega}\cap U)$. Let $U_0$ be an open neighborhood of $x$ and $F\colon \R^n\to\R^n$  be a $\Cont[2]$-diffeomorphism which maps $U_0\cap \overline{\Omega}$ onto $\overline{\R^n_+}\cap V_0$ for some open set $V_0$. Moreover, let $\psi\in \Cont[\infty]_{0}(U_0)$ with $\psi\equiv 1$ on some neighborhood $U_1\Subset U_0$ of $x$, let $V_1$ be an open set such that $V_1\cap \overline{\R^n_+}=F(U_1\cap \overline{\Omega})$ and let $F^{\ast} l(x)= l(F(x))$ denote the pull-back of $l$ by $F$. For $\varphi\in \Cont[\infty]_{0}(\R^n_+)$  we obtain that $v:=  F^{\ast,-1} (\psi u)\in H^{\frac{\alpha}2}_0(\R^n_+)$ solves
    \begin{eqnarray*}
        \lefteqn{  \widetilde{\mathcal{E}} (v,\varphi)+ \int_{\R^n_+} \phi'(v)\varphi \omega \, dx}\\
        &=& \mathcal{E} (\psi u,F^{\ast}(\varphi)) = \mathcal{E} (u,\psi F^{\ast}(\varphi)) + ([\LL,\psi]u,F^{\ast}(\varphi))_{L^2(\Omega)}\\
        &=& (g,\psi F^{\ast}(\varphi))_{L^2(\Omega)} + ([\LL,\psi]u,F^{\ast}(\varphi))_{L^2(\Omega)}\\
        &=& (\tilde{g},\varphi)_{L^2(\R^n_+)} + ([\LL,\psi]u,F^{\ast}(\varphi))_{L^2(\Omega)}
    \end{eqnarray*}
    where
    \begin{align}
        &\widetilde{\mathcal{E}} (\varphi,\psi) = \int_{\R^n_+}\int_{\R^n_+} (\varphi(x)-\varphi(y))(\psi(x)-\psi(y)) \tilde{k}(x,y,x-y)dxdy \label{eq:tildeE}\\
        &\tilde{k}(x,y,z) = k(F^{-1}(x),F^{-1}(y),A(x,y)z)\omega(x)\omega(y),\label{eq:tildek}
    \end{align}
    and
    \begin{align*}
        &A(x,y)= \int_0^1 DF^{-1}((1-s)y+sx)ds,\\
        &\tilde{g}(x)= g(F^{-1}(x))\omega(x),\qquad \omega(x)=\operatorname{det} DF^{-1}(x).
    \end{align*}
    Moreover, $\widetilde{\LL}$ denotes the integral operator associated to $\widetilde{\mathcal{E}}$. It is not difficult to prove that $\tilde{k}\in \mathcal{K}^\alpha(R')$ for some $R'=R'(R,F)$.
    Now all terms on the right-hand side of the equation above define a functional on $L^2(\R^n_+)$ (see~\cite[Lemma~3.6]{Abels2007}). Hence $v \in L^2(\R_+;H^1(\R^{n-1}))$ by the first arguments in the case $\R^n_+$. Choosing now another $\psi\in \Cont[\infty]_{0}(U_1)$ such that $\psi\equiv 1$ on an open neighborhood $U_2\Subset U_1$ of $x$, one obtains that $v:=  F^{\ast,-1} (\psi u)$ solves
    \begin{align*}
        &\widetilde{\mathcal{E}} (v,\varphi) = (\tilde{g},\varphi)_{L^2(\R^n_+))} - (\eta u, [\LL,\psi] F^{\ast}(\varphi))_{L^2(\Omega)} + ([\LL,\psi](1-\eta)u,F^{\ast}(\varphi))_{L^2(\Omega)}\\
        &= (\tilde{g},\varphi)_{L^2(\R^n_+)} - (\tilde{\eta} v, [\widetilde{\LL},\tilde{\psi}] \varphi)_{L^2(\R^n_+)} + ([\LL,\psi](1-\eta)u,F^{\ast}(\varphi))_{L^2(\Omega)}
    \end{align*}
    for all $\varphi\in \Cont[\infty]_{0}(\R^n_+)$, where $\eta \in \Cont[\infty]_{0}(U_1)$ with $\eta\equiv 1$ on $\supp \psi$, and $\tilde{\psi}= F^{\ast,-1}(\psi), \tilde{\eta}=F^{\ast,-1}(\eta)$. Let us replace $\varphi$ by $-h^{-1}\Delta_{j,h}^- \varphi$. We obtain
    \begin{eqnarray*}
        \lefteqn{  \widetilde{\mathcal{E}} (v,-h^{-1}\Delta_{j,h}^- \varphi)+ h^{-1} \int_\Omega \Delta_{j,h}^+(\phi'(u))\varphi \omega \, dx}\\
        &=& (h^{-1}\Delta_{j,h}^{+}\tilde{g} - \tau_h(\phi'(v))h^{-1}\Delta_{j,h}^{+}\omega,\varphi)_{L^2(\R^n_+)} - (h^{-1}\Delta_{j,h}^+(\tilde{\eta}v), [\widetilde{\LL},\tilde{\psi}] \varphi)_{L^2(\R^n_+)}\\
        && {} + (\tilde{\eta}v, h^{-1}[\Delta_{j,h}^-,[\widetilde{\LL},\tilde{\psi}]] \varphi)_{L^2(\R^n_+)} + (h^{-1}\Delta_{j,h}^+F^{\ast,-1}([\LL,\psi](1-\eta)u),\varphi)_{L^2(\R^n_+)}.
    \end{eqnarray*}
    Observe now that
    \begin{eqnarray*}
        \|h^{-1}\Delta_{j,h}^+\tilde{g}- \tau_h(\phi'(v))h^{-1}\Delta_{j,h}^+\omega\|_{L^2(\R^n_+)} &\leqslant& C\left(\|\partial_{x_j}\tilde{g}\|_{L^2(\R^n_+)} + \|\phi'(v)\|_{L^2(\R^n_+)}\right)\\
        &\leqslant&  C\left(\|g\|_{H^1(\Omega)} + \|\phi'(u)\|_{L^2(\Omega)}\right).
    \end{eqnarray*}
    On the other hand, since $\partial_{x_j}(\tilde{\eta}v), \partial_{x_j} \tilde{g}\in L^2(\R^n_+)$, then
    $h^{-1}\Delta_{j,h}^+(\tilde{\eta}v)$ and $h^{-1}\Delta_{j,h}^{+}\tilde{g}$ are bounded in $L^2(\R^n_+)$. Moreover, we have
    \begin{eqnarray*}
        ([{\LL},\psi](1-\eta)u)(x) &=& - \psi(x) (\LL (1-\eta)u)(x) \\
        &=& \int_\Omega \psi(x)(1-\eta(y))u(y) k(x,y,x-y)\, dy
    \end{eqnarray*}
    because of $\psi (1-\eta)\equiv 0$, where $\psi(1-\eta)k\in \Cont[1](\overline{\Omega}\times\overline{\Omega})$ since $\supp \psi \cap \supp (1-\eta)=\emptyset$ and $k(x,y,x-y)$ is continuously differentiable for $x\neq y$. Therefore $([{\LL},\psi](1-\eta)u)\in \Cont[1](\overline{\Omega})$ and
    \begin{equation*}
        h^{-1}\Delta_{j,h}^+F^{\ast,-1}([\LL,\psi](1-\eta)u)\in L^2(\R^n_+)
    \end{equation*}
    is uniformly bounded. Finally,
    \begin{eqnarray*}
        [h^{-1}\Delta_{j,h}^-,[\widetilde{\LL},\tilde{\psi}]] \varphi &=& [\widetilde{\LL},h^{-1}[\Delta_{j,h}^-,\tilde{\psi}]] \varphi + [\tilde{\psi},[\widetilde{\LL},h^{-1}\Delta_{j,h}^-]]\varphi\\
        &=& [\widetilde{\LL},h^{-1}(\Delta_{j,h}^-\tilde{\psi})\tau_{j,-h}]\varphi+[\tilde{\psi},\widetilde{\LL}_{j,h}^-]\varphi,
    \end{eqnarray*}
    where $\widetilde{\LL}_{j,h}^-$ is the integral operator with kernel $h^{-1}(\tilde{k}(x+he_j,y+he_j,z)-\tilde{k}(x,y,z))$ and by (\ref{k-ass-two}) the latter kernel is bounded by $C|z|^{-d-\alpha}$ uniformly in $h>0$. This implies that
    \begin{eqnarray*}
        \|[h^{-1}\Delta_{j,h}^-,[\widetilde{\LL},\tilde{\psi}]] \varphi\|_{L^2(\R^n_+)} &\leqslant & C\|\varphi\|_{H^{\alpha/2}(\Omega)}
    \end{eqnarray*}
    uniformly in $h>0$. Hence choosing $\varphi= h^{-1} \Delta_{j,h}^+ v$ one obtains arguing as in the half-space case that
    $\partial_{x_j} v\in H^{\frac{\alpha}2}(\R^n_+)$, $j=1,\ldots, n-1$. This entails H\"older continuity of $u$ in a neighborhood of $x$. Estimate \eqref{eq:HoelderEstim} thus follows from the
    estimates obtained in this proof.
\end{proof}

\section{Subgradients}\label{subgra}

Let $\phi\colon [a,b]\to \R$  be a continuous function and set $\phi(x)=+\infty$ for $x\not\in [a,b]$. Then fix $\theta\geqslant 0$
and consider the functional
\begin{equation}\label{eq:DefnF}
    F_{\theta} (c) = \frac{\theta}{2} \int_{\Omega} |\nabla c|^2 \sd x + \E(c,c)  + \int_{\Omega} \phi( c(x)) \sd x
\end{equation}
where
\begin{eqnarray*}
    \dom F_0&=& \left\{c\in H^{\alpha/2}(\Omega)\cap L^2_{(m)}(\Omega): \phi(c)\in L^1(\Omega)\right\},\\
    \dom F_\theta&=& H^1(\Omega)\cap \dom F_0\qquad \text{if}\ \theta >0.
\end{eqnarray*}
for a fixed $m\in(a,b)$.
Moreover, let
\begin{equation*}
    \E_\theta (u,v)= \theta \int_\Omega \nabla u\cdot \nabla v \sd x +\E(u,v)
\end{equation*}
for all $u,v\in H^1(\Omega)$ if $\theta>0$ and $u,v\in H^{\alpha/2}(\Omega)$ if $\theta = 0$.

We denote by $\partial F_\theta(c)\colon L^2_{(m)}(\Omega)\to \mathcal{P}(L^2_{(0)}(\Omega))$ the subgradient of $F_\theta$ at $c\in \dom F$ in the sense that  $w\in \partial F_\theta(c)$ if and only if
\begin{equation*}
    (w,c'-c)_{L^2} \leqslant F_\theta(c')-F_\theta(c)\qquad \forall\,c'\in L^2_{(m)}(\Omega).
\end{equation*}
Note that $L^2_{(m)}(\Omega)$ is an \emph{affine} subspace of $L^2(\Omega)$ with tangent space $L^2_{(0)}(\Omega)$. Therefore the standard definition of $\partial F$ for functionals on Hilbert spaces does not apply. But the definition above is the obvious generalization to affine subspaces of Hilbert spaces.

First of all let us prove the following
\begin{lem}\label{lem:LSC}
    Let $\phi\colon [a,b]\to \R$ be a continuous and convex function. Then, for any $\theta \geqslant 0$, $F_{\theta}$ defined as in~\eqref{eq:DefnF} is a proper, lower semicontinuous, convex functional.
\end{lem}

\begin{proof}
    We only need to prove the lower semicontinuity. The case $\theta > 0$ may be handled as in~\cite[Lemma~4.1]{AsymptoticCH}. On the other hand, if $\theta = 0$, let $c_{k} \in \Lp_{(m)}(\Omega)$ be such that $\lim_{k \to \infty} c_{k} = c$ in $\Lp(\Omega)$ and   $\liminf_{k \to \infty}F_{\theta}(c_{k}) < \infty$. By adding a suitable constant to $\phi$, we can reduce to the case $\phi\geq 0$. Up to a subsequence, we can assume that $c_{k} \in \dom F_{\theta}$ and $c_{k} \toweak c^{\ast}$ in $\Hs[\sfrac{\alpha}{2}](\Omega)$. Hence $c_{k} \to c^{\ast}$ in $\Lp(\Omega)$ and almost everywhere in $\Omega$. Thus we get $c = c^{\ast}$. Moreover, Fatou's lemma and the (weak) continuity of $\mathcal{E}$ imply $c \in \dom F_{\theta}$ and $F_{\theta}(c) \leqslant \liminf_{k \to \infty} F_{\theta}(c_{k})$.
\end{proof}

\begin{cor}\label{cor:MaximalMonotone}
    Let $\phi$ and $F_{\theta}$ be as in Lemma~\ref{lem:LSC} and let $m=0$. Then, for any $\theta \geqslant 0$, $\partial F_{\theta}$ is a maximal monotone operator on $H = \Lp_{(0)}(\Omega)$.
\end{cor}

\begin{proof}
    In view of Lemma~\ref{lem:LSC}, this fact follows from Corollary~1.2 and Lemma 1.3 in \cite[Chapter IV]{Showalter}.
\end{proof}

We now state our main result on the following characterization of $\partial F(c)$:
\begin{theorem}\label{thm:Regularity}
    Let $\phi\colon [a,b]\to \R$ be a convex function that is twice continuously differentiable in $(a,b)$ and satisfies $\lim_{x\to a} \phi'(x)=-\infty$, $\lim_{x\to b} \phi'(x) =+\infty $. Moreover, we set $\phi'(x)=+\infty$ for $x\not\in (a,b)$ and let $F_\theta$ be defined as in~\eqref{eq:DefnF}. Then
    \begin{eqnarray*}
        \mathcal{D}(\partial F_0) &=&\Big\{ c\in H^\alpha_{\loc}(\Omega)\cap H^{\alpha/2}(\Omega)\cap L^2_{(m)}(\Omega): \phi'(c)\in L^2(\Omega),\exists f\in L^2(\Omega):\\
        & & \quad  \E(c,\varphi) + \int_\Omega \phi'(c)\varphi \, dx= \int_{\Omega} f\varphi \sd x \quad \forall \, \varphi \in H^{\alpha/2}(\Omega)
        \Big\}
    \end{eqnarray*}
    if $\theta=0$ and
    \begin{eqnarray*}
        \mathcal{D}(\partial F_\theta) &=&\Big\{ c\in \Hs[2]_{\loc}(\Omega) \cap \Hs[1](\Omega) \cap L^2_{(m)}(\Omega): \phi'(c)\in L^2(\Omega),\exists f\in L^2(\Omega):\\
        & &\quad  \E_\theta(c,\varphi) + \int_\Omega \phi'(c)\varphi \, dx= \int_{\Omega} f\varphi \sd x \quad\, \forall\,\varphi \in \Hs(\Omega) \Big\}
    \end{eqnarray*}
    if $\theta > 0$ as well as
    \begin{equation*}
        \partial F_\theta (c) = -\theta  \Delta c + \LL c + P_0\phi'(c)\quad \text{in}\ \mathcal{D}'(\Omega) \qquad \text{for $\theta \geqslant 0$.}
    \end{equation*}
    Moreover, the following estimates hold
    \begin{alignat}{1}\label{eq:DomEstim}
          \theta \|c\|_{H^1}^2+ \|c\|_{H^{\alpha/2}}^2+ \|\phi'(c)\|_2^2&\leqslant  C\left(\|\partial F_\theta(c)\|_2^2 + \|c\|_2^2+1\right)\\\nonumber
        \int_\Omega\int_\Omega (\phi'(c(x))-\phi'(c(y)))&(c(x)-c(y))k(x,y,x-y)\sd x\sd y\\\nonumber
        &\leqslant  C\left(\|\partial F_\theta(c)\|_2^2 + \|c\|_2^2+1\right)\\\nonumber
        \theta \int_\Omega \phi''(c)|\nabla c|^2 \sd x &\leqslant C\left(\|\partial F_\theta(c)\|_2^2 + \|c\|_2^2+1\right)
    \end{alignat}
    for some constant $C>0$ independent of $c\in \mathcal{D}(\partial F_\theta) $ and $\theta \geqslant 0$.
\end{theorem}

\medskip\noindent
\begin{proof}

We will follow the same strategy as in~\cite[Theorem~4.3]{AsymptoticCH}. Let us introduce first some technical tools and simplifications. If we replace $c(x)$ by $\bar{c}(x) = c(x)-m$ and $\phi$ by $\bar{\phi}(c)=\phi(c+m)$, we can assume w.l.o.g. that $m=0\in (\bar{a},\bar{b})$. Moreover, replacing $\phi(c)$ by $\bar{\phi}(c)= \phi(c)+ b_1 c(x)+ b_2$, $b_j\in\R$ means changing $F$ only by an affine linear functional, for which the subgradient is trivial. In this way,  we may also assume that $\phi'(0)=\phi(0)=0$. Furthermore, we define $\phi_+(c)=\phi(c)$ if $c>0$, $\phi_+(c)=0$ if $c \leqslant 0$ and $\phi_-(c)=\phi(c)-\phi_+(c)$. Then $\phi_\pm\colon \R\to \R\cup \{+\infty\}$ are convex functions, which are continuously differentiable in $(a,b)$.

In the following, we would like to evaluate the directional derivative of $F_{\theta}(c)$ along $\phi'(c)$. Formally, this requires the estimate of $\|\phi'(c)\|_2$, but we cannot do this directly due to the singular behavior of $\phi$. Therefore we approximate $\phi'_+$ (and analogously $\phi'_-$) from below by a sequence $f_n^+$ of smooth potentials  as follows. Since $\phi'$ is continuous and monotone, $\phi'(0)=0$, and $\lim_{c\to b} \phi'(c) =+\infty $, for every $n\in\N$ sufficiently large there is some $c_n\in (\frac{b}2,b)$ such that $\phi'(c_n)=n$. Therefore we can define
\begin{equation*}
    f_n^+(c) =
    \begin{cases}
        \phi'(c) & \text{for}\ c\in [\frac{b}2,c_n)\\
        n+\phi''(c_n)(c-c_n) & \text{for}\ c \geqslant c_n\\
        0 & \text{for}\ c \leqslant 0
    \end{cases}
\end{equation*}
for $c\not\in (0,\frac{b}2)$. Moreover, we can extend $f_n^+$ to $\R$ such that $f_n^+\colon\R\to\R$ are $C^1$-functions with $0\leqslant f_n^+\leqslant \phi_+'$ and with first derivative bounded by $M_n:= \sup_{0\leqslant x\leqslant c_n} \phi''(x)$.

We have to work in the subspace $L^2_{(0)}(\Omega)$. Then we will use ``bump functions'' supported in suitable sets to correct the mean value of functions. For this let $c\in H^1_{(0)}(\Omega)$ be fixed and let $I\subset [a,b]$ be an interval such that $|\{c(x) \in I\}|>0$. We say that $\varphi$ is a \emph{bump function supported in $\{ c\in I \}$} if $\varphi\in H^1(\Omega)\cap L^\infty(\Omega)$, $\varphi\geqslant 0$, $\varphi(x)=0$ if $c(x)\not\in I$ and if $m(\varphi)=1$. Such a function can be constructed as follows. Choose a smooth function $\psi\colon \R\to [0,1]$ with bounded first derivative such that $\psi(s)=0$ if $s\not\in I$ and $\psi(s)>0$ otherwise. Then $\varphi= \frac{\psi(c)}{m(\psi(c))}$ has the stated properties. Furthermore, we note that, if $I=[a,a']$ with $a'\in (a,b)$, then we can choose $\psi$ such that $\psi'(s)\leqslant 0$. This implies that the constructed function $\varphi$ has the property
\begin{equation}\label{eq:SignDt}
    (\nabla c, \nabla \varphi)_{L^2(\Omega)} = \frac1{m(\psi(c))}\int_\Omega \psi'(c) |\nabla c|^2 \sd x \leqslant 0\\
\end{equation}
as well as
\begin{align}\label{eq:SignDt2}
    \mathcal{E}(c, \varphi) &= \frac{1}{2 \mean{\psi(c)}} \int_{\Omega} \int_{\Omega} (c(x) - c(y))(\psi(c(x)) - \psi(c(y))) k(x, y, x-y) \, \mathrm{d}x \mathrm{d}y \notag\\
    &= \frac{1}{2 \mean{\psi(c)}} \int_{\Omega} \int_{\Omega} \psi'(\xi) (c(x) - c(y))^{2} k(x, y, x-y) \, \mathrm{d}x \mathrm{d}y \leqslant 0
\end{align}
where $\xi(x,y)$ is a measurable function which is bounded above and below by $\max \{ c(x), c(y) \}$ and $\min \{ c(x), c(y) \}$ respectively. Given such a bump function $\varphi$, we define  $M_\varphi \colon L^2(\Omega)\to H^1(\Omega)\cap L^\infty(\Omega)$ by
\begin{equation*}
    (M_\varphi f)(x) = m(f)\varphi,\qquad f\in L^2(\Omega).
\end{equation*}
Then $f-M_\varphi f \in L^2_{(0)}(\Omega)$ and
\begin{equation}\label{eq:EstimPhiTilde}
    \|M_\varphi f\|_{H^1} \leqslant C \left|\int_\Omega f(x) \sd x\right|\qquad \forall\,\ f\in L^2(\Omega).
\end{equation}

Observe now that
\begin{equation*}
    \left|\left\{c(x)-a \geqslant t\right\}\right|\leqslant \frac1{t} \int_{\Omega} (c(x)-a)\sd x = \frac{|a||\Omega|}t
\end{equation*}
for $t>0$ since $c\in L^2_{(0)}(\Omega)$. This implies that $|\{c < \frac{b}2 \}| \geqslant \frac{b}{b+2|a|}|\Omega|>0$. Hence the interval $I=[a,\frac{b}2)$ is admissible for the construction of bump functions supported in $\{c\in (a,\frac{b}2)\}$.

After these preliminary considerations, let $c \in \mathcal{D}(\partial F_{\theta})$. We define $\tilde{c}_t(x)$, $0<t\leqslant \frac2{M_n}$, $x\in\Omega$, as solution of
    \begin{equation}
        \tilde{c}_t(x) = c(x) - t f_n^+(\tilde{c}_t(x)),
    \end{equation}
    which exists by the contraction mapping principle. Then $\tilde{c}_t(x)=c(x)$ if $c(x)< 0$ since $f_n^+(\tilde{c}_t(x))=0$ in this case. Moreover, we have that $0\leqslant \tilde{c}_t(x)= c(x)-tf_n^+(\tilde{c}_t(x))\leqslant c(x)$ if $c(x)\geqslant 0$. More formally, $\tilde{c}_t$ can be expressed in the form $\tilde{c}_t (x) = F_t^n(c(x))$, where $F_t^n\colon [a,b]\to [a,b]$ is a continuous differentiable mapping with $F_t^n(x) \to x$, $(F_t^n)'(x) \to 1$ as $t\to 0+$ uniformly in $[a,b]$. Hence, if $\theta > 0$, $\tilde{c}_t\in H^1(\Omega)$ and $\tilde{c}_t\to_{t\to 0} c$ in $H^1(\Omega)$ and almost everywhere. If else $\theta = 0$, we deduce $\tilde{c}_{t} \in \Hs[\sfrac{\alpha}{2}](\Omega)$ and $\tilde{c}_{t} \to_{t \to 0} c$ in $\Hs[\sfrac{\alpha}{2}](\Omega)$ and almost everywhere.

    Since, in general, $\tilde{c}_t(x)\not\in L^2_{(0)}(\Omega)$, we set  $c_t= \tilde{c}_t+ tM_\varphi (f_n^+(\tilde{c}_t))$, where $\varphi$ is a bump function supported in $\{c(x)<\frac{b}2\}$ satisfying~\eqref{eq:SignDt} and~\eqref{eq:SignDt2}. Then $c_t\in L^2_{(0)}(\Omega)$. Furthermore, $c_t(x)=\tilde{c}_t(x)$ and $f_n^+(c_t(x))=f_n^+(\tilde{c}_t(x))$ if $c(x)>\frac{b}2$ and $c_t(x)= \tilde{c}_t(x) + t M_\varphi (f_n^+(c_t))\in [a,\frac34b]$ if $c(x)\leqslant \frac{b}2$ and if $0<t<\frac{b}{4M_n'}$ where $M_n'= \sup_{0\leqslant t\leqslant b} f_n^+(t)\|\varphi\|_\infty$. We denote $d_t= M_\varphi(f_n^+(\tilde{c}_t))$.

    We now assume that $w\in \partial F_{\theta}(c)$. Thus we have
    \begin{equation*}
        F_{\theta}(c) - F_{\theta}(c_t) \leqslant t(w,f_n^+(\tilde{c}_t)- d_t)_{L^2(\Omega)}.
    \end{equation*}
    Moreover, if $t>0$ is sufficiently small, a direct computation involving the definition of $F$ and the above construction gives
    \begin{align*}
        &F_{\theta}(c) - F_{\theta}(c_t)\\
        =& \int_{\Omega}(\phi(c(x))-\phi(c_t(x))) \sd x + \theta t(\nabla c,\nabla f_n^+(\tilde{c}_t))_{L^2} - \theta tm(f_n^+(\tilde{c}_t))(\nabla c,  \nabla \varphi)_{L^2}\\
        & \quad {} - \theta \frac{t^2}2 \|\nabla (f_n^+(\tilde{c}_t)-d_t) \|_{L^2}^2 - \mathcal{E}(c-c_{t}, c-c_{t})\\
        & \quad {} + t \int_{\Omega} \int_{\Omega} (f_{n}^{+}(\tilde{c}_{t}(x)) - f_{n}^{+}(\tilde{c}_{t}(y)))(c(x) - c(y)) k(x, y, x-y) \, \mathrm{d}x \mathrm{d}y\\
        & \quad {} - t \mean{f_{n}^{+}(\tilde{c}_{t})} \int_{\Omega} \int_{\Omega} (\varphi(x) - \varphi(y))(c(x) - c(y)) k(x, y, x-y) \, \mathrm{d}x \mathrm{d}y \\
        \end{align*}
    Therefore, we deduce
    \begin{align*}
    &F_{\theta}(c) - F_{\theta}(c_t)\\
        \geqslant & t\int_{\{c(x) >\frac{b}2\}} \phi'(c_t(x)) f_n^+(c_t(x))\sd x+ t\int_{\{c(x) \leqslant \frac{a}2\}} \left(\phi(c(x))-\phi(c(x) + td_t)\right)\sd x\\
        & \quad {} + \int_{\{\frac{a}2\leqslant c(x) \leqslant \frac{b}2\}} \left(\phi(c(x))-\phi(\tilde{c}_t(x) + td_t)\right)\sd x + \theta t(\nabla c,\nabla f_n^+(\tilde{c_t})) \\
        & \quad {} + t \int_{\Omega} \int_{\Omega} (f_{n}^{+}(\tilde{c}_{t}(x)) - f_{n}^{+}(\tilde{c}_{t}(y))) (c(x) - c(y)) k(x, y, x-y) \, \mathrm{d}x \mathrm{d}y \\
        & \quad {} - \mathcal{E}(c-c_{t}, c-c_{t})- \theta \frac{t^2}2 \| \nabla (f_n^+(\tilde{c}_t)-d_t) \|_{L^2}^2.
    \end{align*}
    Hence
    \begin{align*}
    &F_{\theta}(c) - F_{\theta}(c_t)\\
        \geqslant & t\int_{\{c(x) >\frac{b}2\}} f_n^+(c_t(x))^2\sd x + \theta t(\nabla c,\nabla f_n^+(\tilde{c}_t)) - \theta \frac{t^2}2 \| \nabla (f_n^+(\tilde{c}_t)-d_t) \|_{L^2}^2\\
        & \quad {} + \int_{\{\frac{a}2\leqslant c(x) \leqslant \frac{b}2\}} \left(\phi(c(x))-\phi(c(x) + td_t)\right)\sd x - \mathcal{E}(c-c_{t}, c-c_{t})\\
        & \quad {} + t \int_{\Omega} \int_{\Omega} (f_{n}^{+}(\tilde{c}_{t}(x)) - f_{n}^{+}(\tilde{c}_{t}(y))) (c(x) - c(y)) k(x, y, x-y) \, \mathrm{d}x \mathrm{d}y,
    \end{align*}
    where we have used that $\phi(c)-\phi(c_t)\geqslant \phi'(c_t)(c-c_t)$ and $c_t<c$ if $c>\frac{b}2$, $\phi'(c_t)\geqslant f_n^+(c_t)$  as well as \eqref{eq:SignDt}, \eqref{eq:SignDt2} and $\phi(c)-\phi(c + td_t)\geqslant 0$ if $c\leqslant \frac{a}2$ and $t \leqslant \frac{a}{2M'_{n}}$. Hence we deduce
    \begin{align*}
        &(w,f_n^+(\tilde{c}_t)-d_t)_{L^2(\Omega)}\\
        \geqslant &\int_{\{c(x) > \frac{b}2\}} f_n^+ (c_t(x))^2\sd x + \theta (\nabla c,\nabla f_n^+(\tilde{c}_t)) - \theta \frac{t}{2} \| \nabla (f_n^+(c_t)-d_t) \|_{L^2}^2\\
        & \quad {} + \int_{\Omega} \int_{\Omega} (f_{n}^{+}(\tilde{c}_{t}(x)) - f_{n}^{+}(\tilde{c}_{t}(y)))(c(x) - c(y)) k(x, y, x-y) \, \mathrm{d}x \mathrm{d}y\\
        &\quad - \frac{t}{2} {\mathcal{E}(f_{n}^{+} (\tilde{c}_{t}) - d_{t}, f_{n}^{+} (\tilde{c}_{t}) - d_{t})} \\
        &\quad + \int_{\{\frac{a}{2} \leqslant c(x) \leqslant \frac{b}{2}\}} \frac{1}{t} \left( \phi(c(x))-\phi(\tilde{c}_{t}(x) + t d_{t}) \right) \sd x,
    \end{align*}
    which yields for $t \to 0$
    \begin{align*}
        & (w,f_n^+(c)-M_\varphi(f_n^+(c)))_{L^2(\Omega)}\\
        \geqslant & \int_{\{c(x) \geqslant \frac{b}2\}} f_n^+(c(x))^2\sd x \\
        & {}+ \int_{\{\frac{a}{2} \leqslant c(x) \leqslant \frac{b}{2}\}} \phi'(c(x))(f_{n}^{+}(c(x))-M_\varphi(f_{n}^{+}(c)))\sd x\\
        & {}+ \theta \Ltwoprod{\nabla c}{\nabla f_n^+(c)} \\
        & {}+ \int_{\Omega} \int_{\Omega} (f_{n}^{+}(c(x)) - f_{n}^{+}(c(y)))(c(x) - c(y)) k(x, y, x-y) \, \mathrm{d}x \mathrm{d}y
    \end{align*}
    since $\lim_{t\to 0}\tilde{c}_t = c$ in $H^1(\Omega)$ for $\theta > 0$ (in $\Hs[\sfrac{\alpha}{2}](\Omega)$ for $\theta = 0$) and almost everywhere and since $\phi(c)$ is continuously differentiable in $[\frac{a}2,\frac34 b]$. Observe now that
    \begin{gather*}
        \theta \Ltwoprod{\nabla c}{\nabla f_n^+(c)} = \theta \int_{\Omega}(f^+_n)'(c(x))|\nabla c(x)|^2\sd x \geqslant 0,\\
        \begin{alignat*}{1}
            &\int_{\Omega} \int_{\Omega} (f_{n}^{+}(c(x)) - f_{n}^{+}(c(y)))(c(x) - c(y)) k(x, y, x-y)) \, \mathrm{d}x \mathrm{d}y\\
            =& \int_{\Omega} \int_{\Omega} (f_{n}^{+})'(\xi) (c(x) - c(y))^{2} k(x, y, x-y) \, \mathrm{d}x \mathrm{d}y \geqslant 0,
        \end{alignat*}
    \end{gather*}
    where $\xi(x,y)$ is a measurable function which is bounded above and below by $\max \{ c(x), c(y) \}$ and $\min \{ c(x), c(y) \}$, respectively, and use the fact that $\|M_\varphi(f_n^+(c))\|_2\leqslant C \|f_n^+(c)\|_2$ on account of~\eqref{eq:EstimPhiTilde}. Therefore, we get
    \begin{align*}
        &\|f_n^+(c)\|_{L^2(\Omega)}^2 + \theta \int_{\Omega}(f^+_n)'(c(x))|\nabla c(x)|^2\sd x\\
        &\quad {} + \int_{\Omega} \int_{\Omega} (f_{n}^{+}(c(x)) - f_{n}^{+}(c(y)))(c(x) - c(y)) k(x, y, x-y) \, \mathrm{d}x \mathrm{d}y \\
        \leqslant & C\left(\|w\|_{L^2(\Omega)}^2 + \int_{\{\frac{a}2 \leqslant c(x) \leqslant \frac{b}2\}} |\phi'(c(x))|^2\sd x + 1 \right)\\
        \leqslant & C'\left(\|w\|_{L^2(\Omega)}^2 + \int_{\Omega} |c(x)|^2\sd x + 1 \right)
    \end{align*}
    by Young's inequality and letting $n\to \infty$  we infer
    \begin{multline}\label{eq:EstimPhiPrim}
        \|\phi'_+(c)\|_{L^2(\Omega)}^2 + \theta \int_{\Omega}\phi_+''(c(x))|\nabla c(x)|^2\sd x\\
        {}+ \int_{\Omega} \int_{\Omega} (\phi_{+}'(c(x)) - \phi_{+}'(c(y))) (c(x) - c(y)) k(x, y, x-y) \, \mathrm{d}x \mathrm{d}y\\
        \leqslant C\left(\|w\|_{L^2(\Omega)}^2 + \Lpnorm{c}^2 + 1 \right)
    \end{multline}
    by Fatou's lemma. By symmetry the same is true for $\phi_-$ instead of $\phi_+$ and therefore also for $\phi$.

    In particular, $\phi'(c)\in L^2(\Omega)$ implies $c(x)\in (a,b)$ almost everywhere in $\Omega$. Thus $|\{c(x)\in (a+\delta,b-\delta)\}|>0$ for sufficiently small $\delta>0$. Because of this, we can use a bump function $\varphi$ supported in $\{c(x) \in (a+\delta,b-\delta)\}$ for some fixed $\delta>0$. Moreover, let $\psi_M\colon \R \to [0,1]$, $M\in\N$, be smooth functions such that $\psi_M(s)=0$ if $|s|\geqslant M+1$, $\psi_M(s)=1$ if $|s|\leqslant M$, and $|\psi'_M(s)| \leqslant 2$. Set $\chi_M(x)= \psi_M(\phi'(c(x)))$. Then $\chi_M\in H^1(\Omega)$ and $\chi_M(x)=0$ if $\phi'(c(x))\geqslant M+1$. Moreover, $\chi_M \to_{M\to \infty} 1 $ almost everywhere and in $L^p(\Omega)$, $1\leqslant p<\infty$. On the other hand, we have
    \begin{align*}
        &(\nabla c, \nabla (\chi_M\psi))_{L^2(\Omega)} = (\nabla c, \chi_M \nabla \psi)_{L^2(\Omega)} \\
        &+ \int_{\Omega} \phi''(c(x))|\nabla c(x)|^2 \psi(x) \psi_M'(\phi'(c(x)) \sd x
    \end{align*}
    for all $\psi\in C^\infty(\ol{\Omega})$ if $\theta > 0$. Since $\phi''(c)|\nabla c|^2\in L^1(\Omega)$ due to~\eqref{eq:EstimPhiPrim} and $\psi_M'(\phi'(c(x)))\to_{M\to \infty} 0$ almost everywhere, we conclude
    \begin{equation}\label{eq:limMInfty}
        \lim_{M\to \infty}(\nabla c, \nabla (\chi_M\psi))_{L^2(\Omega)} = (\nabla c,\nabla \psi)_{L^2(\Omega)}\qquad \forall\, \psi\in C^\infty(\ol{\Omega})
    \end{equation}
    as soon as $\theta > 0$. Analogously, for all $\psi \in \Cont[\infty](\closure{\Omega})$, we also have
    \begin{align*}
        &\mathcal{E}(c, \chi_{M} \psi)\\
        =& \frac{1}{2} \int_{\Omega} \int_{\Omega} (c(x) - c(y))(\chi_{M}(x)\psi(x) - \chi_{M}(y)\psi(y)) k(x, y, x-y) \, \mathrm{d}x \mathrm{d}y\\
        =& \frac{1}{2} \int_{\Omega} \int_{\Omega} (c(x) - c(y)) \chi_{M}(x) (\psi(x) - \psi(y)) k(x, y, x-y) \, \mathrm{d}x \mathrm{d}y\\
        & + \frac{1}{2} \int_{\Omega} \int_{\Omega} (c(x) - c(y))(\chi_{M}(x) - \chi_{M}(y)) \psi(y) k(x, y, x-y) \, \mathrm{d}x \mathrm{d}y.
    \end{align*}
    Recalling that $\phi'$ is monotone and $|\psi_M'(s)|\leq 2$, for any positive and bounded $\psi$ we have
    \begin{multline*}
        \left|(c(x) - c(y))(\chi_{M}(x) - \chi_{M}(y)) \psi(y) k(x, y, x-y)\right| \\
        \leqslant 2 (c(x) - c(y))(\phi'(c(x)) - \phi'(c(y))) \psi(y)  k(x, y, x-y)
    \end{multline*}
    and therefore we deduce
    \begin{equation*}
        (c(x) - c(y)) (\chi_{M}(x) - \chi_{M}(y)) \psi(y) k(x, y, x-y) \in \Lp[1](\Omega \times \Omega).
    \end{equation*}
    Moreover, $\chi_{M}(x) - \chi_{M}(y) \to 0$ almost everywhere in $\Omega \times \Omega$ and $\chi_{M}(x) \to 1$ almost everywhere in $\Omega$, so that by the dominated convergence theorem we obtain, for all $\theta \geqslant 0$,
    \begin{equation}\label{eq:limMInfty2}
        \lim_{M \to \infty} \mathcal{E}(c, \chi_{M} \psi) = \mathcal{E}(c, \psi) \qquad \forall\, \psi \in \Cont[\infty](\closure{\Omega}).
    \end{equation}

   We now set $c_t^M=c - t\chi_M \psi+ t M_\varphi(\chi_M\psi)$, $\psi\in \Cont[\infty](\ol{\Omega})$, $t>0$, $M\in\N$. Then $c_t^M\in \dom F_{\theta}$ for sufficiently small $t>0$ (depending on $M$)  and
   \begin{align*}
        & t \Ltwoprod{w}{\chi_M\psi- M_\varphi (\chi_M\psi)}\\
        \geqslant & F_{\theta}(c) - F_{\theta}(c_t^M)\\
        = & \int_{\{\phi'(c(x))\leqslant M+1\}} \left(\phi(c(x))-\phi(c_t^M(x))\right) \sd x + \theta t (\nabla c,\nabla (\chi_M\psi-M_\varphi(\chi_M\psi)))_{L^2}\\
        &- \frac{\theta t^2}{2} \Lpnorm{ \nabla (\chi_M \psi - M_\varphi(\chi_M \psi)) }^2 \\
        &- t^{2} \mathcal{E}(\chi_{M} \psi - M_{\varphi}(\chi_{M} \psi), \chi_{M} \psi - M_{\varphi}(\chi_{M} \psi))\\
        &+ t \int_{\Omega} \int_{\Omega} (\chi_{M}(x)\psi(x) - \chi_{M}(y) \psi(y)) (c(x) - c(y)) k(x, y, x-y) \, \mathrm{d}x \mathrm{d}y\\
        &- t \, \mean{\chi_{M}\psi} \int_{\Omega} \int_{\Omega} (\varphi(x) - \varphi(y))(c(x) - c(y)) k(x, y, x-y) \, \mathrm{d}x \mathrm{d}y.
    \end{align*}
    Dividing by $t$ and passing to the limit $t \to 0$, we conclude
    \begin{align*}
        &\Ltwoprod{w}{\chi_M\psi- M_\varphi (\chi_M\psi)}\\
        \geqslant & \int_\Omega \phi'(c(x))(\chi_M \psi-M_\varphi (\chi_M\psi)) \sd x + \theta (\nabla c,\nabla (\chi_M\psi-M_\varphi (\chi_M\psi)))_{L^2}\\
        & \quad {} + \int_{\Omega} \int_{\Omega} (\chi_{M}(x) \psi(x) - \chi_{M}(y) \psi(y)) (c(x) - c(y)) k(x, y, x-y) \, \mathrm{d}x \mathrm{d}y\\
        & \quad {} - \mean{\chi_{M} \psi} \int_{\Omega} \int_{\Omega} (\varphi(x) - \varphi(y))(c(x) - c(y)) k(x, y, x-y) \, \mathrm{d}x \mathrm{d}y
    \end{align*}
    for all $\psi\in C^\infty(\ol{\Omega})$. Replacing $\psi$ by $-\psi$, we obtain equality in the above inequality. Finally, letting $M \to \infty$, we get
    \begin{equation*}
        (w,\psi)_{L^2(\Omega)} = (\phi'(c),\psi)_{L^2(\Omega)} + \theta (\nabla c,\nabla \psi)_{L^2(\Omega)} + \mathcal{E}(c, \psi)
    \end{equation*}
    for all $\psi \in \Cont[\infty](\closure{\Omega})$, $\mean{\psi} = 0$, where we have used~~\eqref{eq:EstimPhiTilde},~\eqref{eq:limMInfty},~\eqref{eq:limMInfty2}, and
    \begin{equation*}
        \lim_{M\to \infty}\int_{\Omega} \chi_M \psi \sd x = \lim_{M\to \infty}\int_{\Omega} (\chi_M-1) \psi \sd x = 0\qquad \text{if}\  m(\psi) =0.
    \end{equation*}

    Hence $-\theta \lapl_{N} c + \mathcal{L}c = w - P_{0} \phi'(c)\in L^2_{(0)}(\Omega)$. Using Lemma~\ref{L:regularity_regularization} we deduce $\theta^{\sfrac{1}{2}} c \in \Hs[2]_{\loc}(\Omega)$, $c \in \Hs[\alpha]_{\loc}(\Omega) \cap \Hs[\sfrac{\alpha}{2}](\Omega)$ and
    \begin{equation*}
 \theta^{\frac12}\norm{H^1}{c}+ \norm{\Hs[\sfrac{\alpha}{2}]}{c} \leqslant \Const \Lpnorm{\phi'} + \Const \Lpnorm{w}.
   \end{equation*}
 Using this and~\eqref{eq:EstimPhiPrim}, we obtain~\eqref{eq:DomEstim}. Moreover, the previous observations imply that $\partial F_{\theta}(c)= -\theta \lapl c + \mathcal{L}c + P_{0} \phi'(c)$ is single-valued and $\mathcal{D}(\partial F_\theta)$ is contained in the set on the right-hand side of the identities for $\mathcal{D}(\partial F_\theta)$ (see statement of this theorem).

    Conversely, recalling the definition of subdifferential, the properties of coercive bilinear forms $\Ltwoprod{\nabla u}{\nabla v}$ and $\mathcal{E}(u, v)$ as well as the convexity of $\phi$, it can be easily checked that $-\theta \lapl c + \mathcal{L}c + P_{0} \phi'(c)\in \partial F_\theta(c)$ for any $c$ in the set on the right-hand side of the identities for $\mathcal{D}(\partial F_\theta)$.  This finishes the proof.
\end{proof}

\begin{cor}\label{cor:CharacterisationDom}
    Let $\theta > 0$ and let $F_{\theta}$ be defined as above. Extend $F_{\theta}$ to a functional $\widetilde{F_{\theta}} \colon \Hzero(\Omega)\to \R\cup \{+\infty\}$ by setting $\widetilde{F_{\theta}}(c) = F_{\theta}(c)$ if $c \in \dom F_{\theta}$ and $\widetilde{F_{\theta}}(c) = +\infty$ else. Then $\widetilde{F_{\theta}}$ is a proper, convex, and lower semicontinuous functional, $\partial \widetilde{F_{\theta}}$ is a maximal monotone operator with $\partial \widetilde{F_{\theta}} (c) = -\lapl_{N} \partial F_{\theta}(c)$ and
    \begin{equation}\label{eq:CharacterisationDom}
        \mathcal{D}(\partial \widetilde{F_{\theta}}) = \{c \in \mathcal{D}(\partial F_{\theta}) \mid \partial F_{\theta}(c) = -\theta \lapl c + \mathcal{L}c + P_{0} \phi'(c)  \in \Hone(\Omega) \}.
    \end{equation}
\end{cor}

\begin{proof}
    The lower semicontinuity is proved in the same way as in Lemma~\ref{lem:LSC}.  Then the fact that $\partial \widetilde{F_{\theta}}$ is a maximal monotone operator follows from Corollary~1.2 and Lemma~1.3 in \cite[Chapter IV]{Showalter}.

    First, let $c\in\mathcal{D}(\partial \widetilde{F_{\theta}})$ and $w\in \partial \widetilde{F_{\theta}}(c)$, i.e.,
    \begin{equation}\label{eq:Subgrad}
        (w,c'-c)_{\Hzero} \leqslant \widetilde{F_{\theta}}(c') - \widetilde{F_{\theta}}(c) \qquad \text{for all}\ c'\in \Hzero(\Omega).
    \end{equation}
    Then let $\mu_0 = -\Delta_N^{-1} w$ and choose $c'\in L^2(\Omega)$. Thus we have
    \begin{align*}
        &(\mu_0,c'-c)_{L^2}\\
        = & -(\nabla \mu_0, \nabla \Delta_N^{-1}(c'-c))_{L^2} = (\nabla \Delta_N^{-1} w,\nabla \Delta_N^{-1}(c'-c))_{L^2}\\
        = & (w,c'-c)_{\Hzero} \leqslant \widetilde{F_{\theta}}(c') - \widetilde{F_{\theta}}(c) = F_{\theta}(c') - F_{\theta}(c)
    \end{align*}
    for all $c'\in L^2(\Omega)$. Hence $\mu_0 = -\theta \lapl c + \mathcal{L}c + P_0\phi'(c)\in \mathcal{D}(\partial F_{\theta})$. On the other hand, $\mu_0 = - \lapl_{N}^{-1} w \in \Hone (\Omega)$. This implies that $\partial \widetilde{F_{\theta}}(c)= -\lapl_{N} \partial F_{\theta}(c)$ and
    \begin{equation*}
        \mathcal{D}(\partial \widetilde{F_{\theta}}) \subseteq \left\{ c \in \mathcal{D}(\partial F_{\theta}) \mid \mu_0= -\theta \lapl c + \mathcal{L}c + P_{0} \phi'(c) \in \Hone(\Omega) \right\}.
    \end{equation*}
    Conversely, let $c \in \mathcal{D}(\partial F_{\theta})$ such that $\mu_{0} = -\theta \lapl c + \mathcal{L} c + P_{0} \phi'(c) = \partial F_{\theta}(c)  \in \Hone(\Omega) $. Then one easily verifies that $w= -\Delta_N\mu_0$ satisfies~\eqref{eq:Subgrad} arguing as above. Hence $c\in \mathcal{D}(\partial F_{\theta})$ and~\eqref{eq:CharacterisationDom} follows.
\end{proof}

\section{Proof of Theorem~\ref{thm:Existence}}\label{S:existence}

We first prove the existence of a weak solution. Let us consider the regularized (formal) problem
\begin{equation} \label{eq:CHr}
    \begin{cases}
        \partial_t c_{\theta} = \lapl \mu_{\theta}                                          \quad   & \text{in $\Omega\times (0,\infty)$}\\
        \mu_{\theta} = -\theta \lapl c_{\theta} + \mathcal{L} c_{\theta} + f'(c_{\theta})   \quad   & \text{in $\Omega\times (0,\infty)$}\\
        \partial_\nu \mu_{\theta} = \partial_\nu c_{\theta} =0                                                         \quad   & \text{on $\partial\Omega\times (0,\infty)$}\\
        c_{\theta}|_{t=0} = c_{0 \theta}                                                           \quad   & \text{in $\Omega$}
    \end{cases}
\end{equation}
 where $\theta > 0$ is a (small) positive real number.

Without loss of generality we suppose
\begin{equation}\label{eq:MeanZero}
    \mean{c_{0 \theta}} = \frac1{|\Omega|}\int_\Omega c_{0 \theta} \sd x =0.
\end{equation}
As in the previous section we can reduce to this case by a simple shift. Since~\eqref{eq:MeanZero} and the definition of $\mathcal{L}$ imply that any solution of~\eqref{eq:CHr} as in Theorem~\ref{thm:Existence} satisfies
\begin{equation*}
    \timeder \int_{\Omega} c_{\theta}(x,t) \sd x = \int_{\Omega} \lapl \mu_{\theta} \sd x = 0,
\end{equation*}
we conclude $\mean{c_{\theta}(t)} = 0$ for almost all $t>0$.

Problem~\eqref{eq:CHr} can be formulated as follows (see~\cite[Theorem~1.2]{AsymptoticCH})
\begin{align}
    \dt{c_{\theta}} + \mathcal{A}_{\theta}(c_{\theta}) + \mathcal{B} c_{\theta} &= 0, \qquad t>0, \label{eq:Abstract1}\\
    c_{\theta}|_{t=0} &= c_{0 \theta} \label{eq:Abstract2}
\end{align}
where
\begin{align*}
    \weight{\mathcal{A}_{\theta}(c_{\theta}), \varphi}_{\Hzero,\Hone} &= (\nabla \mu_{\theta}, \nabla \varphi)_{L^2} \qquad\text{with}\; \mu_{\theta} = -\theta \lapl c_{\theta} + \mathcal{L}c_{\theta} + \phi'(c) \\
    \weight{\mathcal{B} c_{\theta}, \varphi}_{\Hzero,\Hone} &= d(\nabla c_{\theta},\nabla \varphi)_{L^2}
\end{align*}
for all $\varphi\in \Hone(\Omega)$ and
\begin{align*}
    \mathcal{D}(\mathcal{A}_{\theta}) &= \left\{ c \in \Hs[2]_{\loc}(\Omega) \cap \Hs(\Omega) \mid c(x) \in [a,b]\; \forall\, x\in \Omega, \phi'(c)\in \Lp(\Omega), \right.\\
    & \quad \quad  \int_{\Omega} \int_{\Omega} (\phi'(c(x)) - \phi'(c(y)))(c(x) - c(y)) k(x, y, x-y) \, \mathrm{d}x \mathrm{d}y < \infty,\\
    & \quad \quad \left. \phi''(c)|\nabla c|^2 \in \Lp[1](\Omega),\; \partial F_{\theta}(c) \in H^1(\Omega),
    \; \partial_\nu c|_{\partial\Omega} =0
\right\}\\
  \mathcal{D}(\B) &= \Hone(\Omega) \subset \Hzero(\Omega).
\end{align*}
In other words
\begin{equation*}
    \mathcal{A}_{\theta}(c) = - \lapl_{N}(-\theta \lapl c - \mathcal{L}c + P_{0} \phi'(c)), \qquad \mathcal{B} c = -d \lapl_{N} c,
\end{equation*}
where $\lapl_{N} \colon \Hone(\Omega)\subset \Hzero(\Omega)\to \Hzero(\Omega)$ is the Laplace operator with Neumann boundary conditions as above, which is considered as an unbounded operator on $\Hzero(\Omega)$. Moreover, the initial datum~$c_{0 \theta}$ appearing in~\eqref{eq:CHr} is a regularization of the given original datum for  problem~\eqref{eq:CH1}-\eqref{eq:CH4}. In particular, we need $c_{0 \theta}$ to satisfy
\begin{equation*}
    \mean{c_{0 \theta}} = \mean{c_{0}}, \quad \lim_{\theta \to 0} \theta \norm{\Hs}{c_{0 \theta}} = 0 \quad \text{and} \quad \lim_{\theta \to 0} c_{0 \theta} = c_{0} \text{ in $\Hs[\sfrac{\alpha}{2}](\Omega)$.}
\end{equation*}
This can be obtained by considering $c_{0 \theta} = \Psi_{\epsilon(\theta)} \ast c_{0}$ where $\Psi_{\epsilon}$ is a suitable mollifier (e.g., a gaussian kernel) and $\epsilon(\theta)$ is chosen to be sufficiently slowly convergent to $0$ if $\theta \to 0$. Finally, we also introduce a suitable regularized energy for system~\eqref{eq:CHr}, namely,
\begin{equation*}
    E_{\theta}(c) \eqdef \mathcal{E}_{\theta}(c, c) + \int_{\Omega} f(c(x)) \, \mathrm{d}x.
\end{equation*}

In order to apply Theorem~\ref{thm:MonotoneLipschitz} for $\theta$ strictly positive we recall that, on account of Corollary~\ref{cor:CharacterisationDom}, $\mathcal{A} = \partial \widetilde{F_{\theta}}$ is a maximal monotone operator with $\widetilde{F_{\theta}}= \varphi_1 + \varphi_2$,
\begin{align*}
    \varphi_1(c) &= \frac{\theta}{2} \int_{\Omega} |\nabla c(x)|^2 \sd x + \mathcal{E}(c,c),\qquad \dom \varphi_1=  \Hone(\Omega), \\
    \varphi_2(c) &= \int_{\Omega} \phi( c(x)) \sd x,\\
    \dom \varphi_2 &= \dom \varphi=\{c\in \Hone(\Omega) \mid \text{$c\in [a,b]$ a.e.\ in $\Omega$} \}
\end{align*}
Obviously, $\varphi_1|_{\Hone(\Omega)}$ is a bounded, coercive quadratic form on $\Hone(\Omega)$.

    We apply Theorem~\ref{thm:MonotoneLipschitz} with $H_1= \Hone(\Omega)$, $H_0 = \Hzero(\Omega)$, $f=0$, and $\varphi_1,\varphi_2$ as above, where we assume that $\phi(c)\geqslant 0$ without loss of generality. As a consequence there exists a unique solution $c\colon [0,\infty)\to H_0$ to~\eqref{eq:Abstract1}-\eqref{eq:Abstract2} such that  $c\in W^1_2(0,T, H_0)\cap L^\infty(0,T; H_1)$, $\varphi(c)\in L^\infty(0,T)$ for every $T>0$ and $c(t)\in \mathcal{D}(\A_{\theta})$ for almost all $t>0$.

    In order to prove the equivalence of~\eqref{eq:energy_identity}, namely
    \begin{equation}\label{eq:energy_identity_regular}
        E_{\theta}(c_{\theta}(t)) + \int_{0}^{t} \Lpnorm{\nabla \mu_{\theta}(s)}^{2} \, \mathrm{d}s = E_{\theta}(c_{0 \theta}) \quad \text{for all $t > 0$,}
    \end{equation}
    we take advantage of the identity
    \begin{equation*}
        E_{\theta}(c_{\theta}(t)) = \widetilde{F_{\theta}}(c_{\theta}(t)) - \frac{d}{2} \Lpnorm{c_{\theta}(t)}^{2}.
    \end{equation*}
    Because of Lemma~4.3 in \cite[Chapter~IV]{Showalter}, we have
    \begin{equation*}
        \timeder \widetilde{F_{\theta}}(c_{\theta}(t)) = (\partial \widetilde{F_{\theta}}(c_{\theta}(t)), \partial_t c_{\theta}(t))_{\Hzero} = -\|\partial_t c_{\theta}(t)\|_{\Hzero}^2 - (\B c_{\theta}(t),\partial_t c_{\theta}(t))_{\Hzero}.
    \end{equation*}
    Moreover, we have
    \begin{align*}
        (\B c_{\theta}(t),\partial_t c_{\theta}(t))_{\Hzero} &= -d (\lapl_N c_{\theta}(t), \partial_t c_{\theta}(t))_{\Hzero} = d (\nabla c_{\theta}(t),\nabla \lapl_N^{-1} \partial_t c_{\theta}(t))_{L^2}\\
        &= -d \weight{\partial_t c_{\theta}(t), c_{\theta}(t) }_{\Hzero,\Hone} = -\frac{d}{2} \timeder \Lpnorm{c_{\theta}(t)}^{2}
    \end{align*}
    due to \cite[Proposition 23.23]{ZeidlerIIa} and $\|\partial_t c_{\theta}(t)\|_{\Hzero} = \|\lapl_{N} \mu_{\theta}(t)\|_{\Hzero} = \|\mu_{\theta}(t)\|_{\Hone}$. Hence an integration over $[0,t]$ yields
    \begin{multline*}
        \mathcal{E}(c_{\theta}(t), c_{\theta}(t)) + \frac{\theta}{2} \Lpnorm{\nabla c_{\theta}(t)}^{2} + \int_{\Omega} f(c_{\theta}(x,t)) \, \mathrm{d}x + \int_{0}^{t} \Lpnorm{\nabla \mu_{\theta}(s)} \, \mathrm{d}s\\
        = \mathcal{E}(c_{0 \theta}, c_{0 \theta}) + \frac{\theta}{2} \Lpnorm{\nabla c_{0 \theta}}^{2} + \int_{\Omega} f(c_{0 \theta}(x)) \, \mathrm{d}x.
    \end{multline*}
    In particular, this implies
    \begin{equation*}
        \partial_t c_{\theta} = \lapl_{N} \mu_{\theta} \in \Bochner{\Lp}{0}{\infty}{\Hzero(\Omega)}, \qquad \theta^{\sfrac{1}{2}} c_{\theta} \in \Bochner{\Lp[\infty]}{0}{\infty}{\Hone(\Omega)}.
    \end{equation*}

    In order to derive higher regularity, we apply $\partial_t^h$ to~\eqref{eq:Abstract1} and take the inner product with $\partial_t^h c_{\theta}$ in $\Hzero(\Omega)$, where $\partial_t^h f(t)=\frac1h(f(t+h)-f(t))$, $t,h>0$. For any $0 < s < t$, this gives
    \begin{align*}
        &\frac{1}{2} \|\partial_t^h c_{\theta}(t)\|_{\Hzero}^2 + \theta \int_s^t (\nabla \partial_t^h c_{\theta}(\tau), \nabla \partial_t^h c_{\theta}(\tau))_{L^2}\sd \tau + \int_{s}^{t} \mathcal{E}(\partial_{t}^{h} c_{\theta}(\tau), \partial_{t}^{h} c_{\theta}(\tau)) \mathrm{d}\tau\\
        \leqslant & d\int_s^t\|\partial_t^h c_{\theta}(\tau)\|_{L^2}^2\sd \tau + \frac12 \|\partial_t^h c_{\theta}(s)\|_{\Hzero}^2\\
        \leqslant & \frac{c_0}2 \int_s^t\|\partial_t^h c_{\theta}(\tau)\|_{\Hs[\sfrac{\alpha}2]_{(0)}}^2 \sd \tau + C \int_s^t\|\partial_t^h c_{\theta}(\tau)\|_{\Hzero}^2\sd \tau + \frac12 \|\partial_t^h c_{\theta}(s)\|_{\Hzero}^2,
    \end{align*}
    where we have used Ehrling's Lemma applied to $\Hs[\sfrac{\alpha}2]_{(0)}(\Omega)\hookrightarrow\hookrightarrow L^2_{(0)}(\Omega)\hookrightarrow \Hzero(\Omega)$ and
    \begin{equation*}
        (\partial_t^h \mathcal{A}_{\theta} (c_{\theta}(\tau)), \partial_t^h c_{\theta}(\tau))_{\Hzero} \geqslant \theta (\partial_t^h c_{\theta}(\tau), \partial_t^h c_{\theta}(\tau))_{H^1} + \mathcal{E}(\partial_{t}^{h} c_{\theta}(\tau), \partial_{t}^{h} c_{\theta}(\tau)).
    \end{equation*}
    Here $c_0>0$ is such that $\mathcal{E}(u,u)\geq c_0 \|u\|_{\Hs[\sfrac{\alpha}2]_{(0)}}^2$. Furthermore, since $\dt{c_{\theta}} \in \Bochner{\Lp}{0}{\infty}{\Hzero(\Omega)}$, there holds
    \begin{equation*}
        \|\partial_t^h c_{\theta}(s)\|_{\Hzero} \leqslant \frac{1}{h} \int_s^{s+h} \|\partial_t c_{\theta}(\tau)\|_{\Hzero} \sd \tau \to_{h\to 0} \|\partial_t c_{\theta}(s)\|_{\Hzero}
    \end{equation*}
    for almost every $s > 0$ and $\|\partial_t^h c_{\theta}\|_{\Bochner{\Lp}{0}{\infty}{\Hzero}} \leqslant \|\partial_t c_{\theta}\|_{\Bochner{\Lp}{0}{\infty}{\Hzero}}$. Hence $\theta \norm{\Bochner{\Lp}{s}{t}{\Hs_{(0)}}}{\partial_t^h c_{\theta}}^2$, $\norm{\Bochner{\Lp}{s}{t}{\Hs[\sfrac{\alpha}{2}]_{(0)}(\Omega)}}{\partial_{t}^{h} c_{\theta}}$ and $\|\partial_t^h c_{\theta}(t)\|_{\Hzero}$ are uniformly bounded in $h>0$, for all $0 < s < t$. On the other hand, we have
    \begin{equation*}
        \partial_t^h c_{\theta} \to_{h\to 0} \partial_t c_{\theta}
    \end{equation*}
    in $\Bochner{\Lp}{0}{\infty}{\Hzero(\Omega)}$. Thus the uniform (w.r.t.\ $h > 0$) bounds on $\partial_t^h c_{\theta}$ yield that $\partial_t c_{\theta} \in \Bochner{\Lp}{s}{t}{\Hs[\sfrac{\alpha}{2}]_{(0)}(\Omega)} \cap \Bochner{\Lp[\infty]}{s}{t}{\Hs[-1]_{(0)}(\Omega)}$ for every $0 < s < t$.

    In order to derive the estimate near $t=0$, we again apply $\partial_t^h$  to~\eqref{eq:Abstract1} and take the inner product with $t\partial_t^h c_{\theta}$. This gives
    \begin{multline*}
        \frac{t}{2} \norm{\Hs[-1]_{(0)}}{\partial_{t}^{h} c_{\theta}(t)}^{2} + \theta \int_{0}^{t} \tau \Lpnorm{\nabla \partial_{t}^{h} c_{\theta}(\tau)}^{2} \, \mathrm{d}\tau + \int_{0}^{t} \tau \mathcal{E}(\partial_{t}^{h} c_{\theta}(\tau), \partial_{t}^{h} c_{\theta}(\tau)) \, \mathrm{d}\tau\\
        \leqslant \Const \int_{0}^{t} (1+ \tau) \Lpnorm{\partial_{t}^{h} c_{\theta}(\tau)}^{2} \, \mathrm{d}\tau.
    \end{multline*}
    Proceeding as above, we get
    \begin{equation*}
        t^{\sfrac{1}{2}} \partial_t c_{\theta} \in \Bochner{\Lp}{0}{1}{\Hs[\sfrac{\alpha}{2}]_{(0)}(\Omega)} \cap \Bochner{\Lp[\infty]}{0}{1}{\Hs[-1]_{(0)}(\Omega)}.
    \end{equation*}
    This implies $\kappa \mu_{\theta} = \kappa \lapl_{N}^{-1} \partial_t c_{\theta} \in \Bochner{\Lp[\infty]}{0}{\infty}{\Hone(\Omega)}$. Thus~\eqref{eq:DomEstim} yields $\kappa \phi'(c_{\theta}) \in \Bochner{\Lp[\infty]}{0}{\infty}{\Lp(\Omega)}$ since $\kappa \partial F_{\theta}(c) = \kappa \mu_{\theta} + \kappa d c_{\theta} \in \Bochner{\Lp[\infty]}{0}{\infty}{\Lp(\Omega)}$. All these norms are uniformly bounded in $\theta\in (0,1]$.

    We are now ready to pass to the limit for $\theta \to 0$ in~\eqref{eq:Abstract1}. Indeed, for any $\theta > 0$ we have proven that there exist (unique) functions $c_{\theta}(t)$ and $\mu_{\theta}(t)$ satisfying
        \begin{equation}\label{eq:CHr_variational}
        \begin{cases}
            \dt{c_{\theta}} = \lapl_N \mu_{\theta}\\
            (\mu_{\theta},\psi)_{L^2} + \theta (\nabla c_{\theta},\nabla \psi)_{L^2} + \mathcal{E}(c_{\theta}, \psi) + (\phi'(c_{\theta}),\psi)_{L^2} = d (c_{\theta},\psi)_{L^2}
        \end{cases}
    \end{equation}
    for all $\psi \in \Hs_{(0)}(\Omega)$ and for almost every $t > 0$. Moreover, from the previous estimates, for all $\theta\in (0,1]$, we have
    \begin{align*}
        c_{\theta}                          &\in \Bochner{\Lp[\infty]}{0}{\infty}{\Hs[\sfrac{\alpha}{2}]_{(0)}(\Omega)}\\
        \theta^{\sfrac{1}{2}}  c_{\theta}   &\in \Bochner{\Lp[\infty]}{0}{\infty}{\Hs(\Omega)}\\
        \kappa \dt{c_{\theta}}              &\in \Bochner{\Lp[\infty]}{0}{\infty}{\Hs[-1]_{(0)}(\Omega)} \cap \Bochner{\Lp}{0}{\infty}{\Hs[\sfrac{\alpha}{2}]_{(0)}(\Omega)}\\
        \mu_{\theta}                        &\in \Bochner{\Lp}{0}{T}{\Hs(\Omega)} \qquad \text{for all $T > 0$}\\
        \kappa \mu_{\theta}                 &\in \Bochner{\Lp[\infty]}{0}{\infty}{\Hs(\Omega)}\\
        \kappa \phi'(c_{\theta})            &\in \Bochner{\Lp[\infty]}{0}{\infty}{\Lp(\Omega)}
    \end{align*}
    where all the bounds deduced are uniform with respect to $\theta$. Therefore, there exists a sequence $\{ \theta_{n} \}_{n \in \mathbb{N}}$, $\theta_{n} \to_{n \to \infty} 0$ such that $c_{\theta_{n}}$, $\mu_{\theta_{n}}$ and $\phi'(c_{\theta_{n}})$ converge weakly (or weakly\textsuperscript{$\ast$}) in the above spaces to $c$, $\mu$ and $\chi$ respectively as $\theta$ vanishes. More precisely, by a suitable diagonal argument on intervals of the form $[0,m]$, we can assume that also $\mu_{\theta_{n}} \to \mu$ in $\Bochner{\Lp}{0}{m}{\Hs(\Omega)}$ for any $m \in \mathbb{N}$. We can easily pass to the limit in the first equation of~\eqref{eq:CHr_variational} deducing
    \begin{equation*}
        \dt{c(t)} = \lapl \mu(t) \qquad \text{in $\Hs[-1]_{(0)}(\Omega)$, for a.e.\ $t > 0$}
    \end{equation*}
    Let $\psi \in \Cont[\infty](\closure{\Omega})$ and let $s > 0$. Thanks to the convergences listed above, for almost any $t > s$ we can pass to the limit for $\theta \to 0$ in the second equation in \eqref{eq:CHr_variational} to find
    \begin{equation*}
    (\mu(t),\psi)_{L^2(\Omega)} + \mathcal{E}(c(t), \psi) + (\chi(t),\psi)_{L^2(\Omega)} = d (c(t),\psi)_{L^2(\Omega)}
    \end{equation*}
    for almost all $t > 0$ since $s$ can be taken arbitrarily small.

    In order to prove the existence of a weak solution for~\eqref{eq:CH1}--\eqref{eq:CH4} on $\mathbb{R}^{+}$, we only have to identify the (weak) limit $\chi = \lim_{n \to \infty} \phi'(c_{\theta_n})$. Let $0 < s < t$ and $m \in \mathbb{N}$ be fixed. Thanks to Aubin-Lions Lemma, $\dt{c_{\theta_{n}}} \in \Bochner{\Lp}{0}{T}{\Hs[-1]_{(0)}(\Omega)}$ and $c_{\theta_{n}} \in \Bochner{\Lp[\infty]}{0}{T}{\Hs[\sfrac{\alpha}{2}]_{(0)}(\Omega)}$ uniformly in $n$ for all $T > 0$ imply the convergence $c_{\theta_{n}} \to c_{\theta_{n}}$ (up to a subsequence) in $\Cont([0,T);\Lp_{(0)}(\Omega))$ for any $T > 0$. Therefore, we deduce $c_{\theta_{n}}(t) \to c(t)$ almost everywhere in $\Omega$. On the other hand, thanks to Egorov's theorem, there exists a set $\Omega_{m} \subset \Omega$ such that $|\Omega_{m}| \geqslant |\Omega| - \tfrac{1}{2m}$ and on which $c_{\theta_{n}} \to c$ uniformly. We now use the (uniform with respect to $\theta > 0$) estimate on $\phi'(c_{\theta_{n}}(t))$ in $\Lp(\Omega)$. By definition, the quantity
    \begin{equation*}
        M_{\delta, n} \eqdef \left| \left\{ x \in \Omega \mid |c_{\theta_{n}}(x)| > 1-\delta \right\} \right|
    \end{equation*}
    is decreasing in $\delta$ for all $n \in \mathbb{N}$. Since $\phi'(y)$ is unbounded for $y \to \pm 1$, we set
    \begin{equation*}
        c_{\delta} \eqdef \inf_{|c| \geqslant 1-\delta} |\phi'(c)| \to_{\delta\to 0} \infty,
    \end{equation*}
    and we have the uniform Tchebychev inequality
    \begin{equation*}
        \int_{\Omega} |\phi'(c_{\theta_{n}})|^{2}\, \mathrm{d} x \geqslant c_{\delta}^{2} |M_{\delta, n}|.
    \end{equation*}
    From the uniform (with respect to $\theta$) estimate on the norm of $\phi'(c_{\theta_{n}})$ in $\Lp(\Omega)$ we obtain
    \begin{equation*}
        |M_{\delta, n}| \to 0 \qquad \text{for $\delta \to 0$,\, uniformly in $n \in \mathbb{N}$.}
    \end{equation*}
    Therefore, we deduce
    \begin{equation*}
        \lim_{\delta \to 0} \left| \left\{ x \in \Omega \mid |c_{\theta_{n}}(x)| > 1-\delta \right\} \right| = 0 \qquad \text{uniformly in $n \in \mathbb{N}$.}
    \end{equation*}
    Thus there exists $\delta = \delta(m)$, independent of $n$, such that
    \begin{equation*}
        \left| \left\{ x \in \Omega \mid |c_{\theta_{n}}(x)| > 1-\delta \right\} \right| \leqslant \frac{1}{2m} \qquad \forall n \in \mathbb{N}.
    \end{equation*}
    Consider now $N \in \mathbb{N}$ so large that by uniform convergence we have $|c_{\theta_{n}} - c| < \tfrac{\delta}{2}$, $\forall n > N$ on $\Omega_{m}$ and let $\Omega_{m}' \subset \Omega_{m}$ be defined by
    \begin{equation*}
        \Omega_{m}' \eqdef \Omega_{m} \cap \{ x \in \Omega \mid |c_{\theta_{N}}(x)| < 1-\delta \}.
    \end{equation*}
    By the above construction we immediately deduce that $|\Omega_{m}'| > |\Omega| - \tfrac{1}{m}$ and that $|c_{\theta_{n}}(x)| < 1 - \tfrac{\delta}{2}$ for all $n \geqslant N$ and for all $x \in \Omega_{m}'$. Therefore, by the regularity assumptions on the potential $\phi'$ we deduce that $\phi'(c_{\theta_{n}}) \to \phi'(c)$ uniformly on $\Omega_{m}'$. Since $m$ and $s$ are arbitrary we have $\phi'(c_{\theta_{n}}) \to \phi'(c)$ almost everywhere in $\Omega \times \mathbb{R}^{+}$. Finally, the uniqueness of weak and pointwise limits gives $\chi = \phi'(c)$ as claimed\footnote{Let $\{ f_{n} \}_{n \in \mathbb{N}}$ be a sequence of functions such that $f_{n} \toweak f$ in $\Lp[p](\Omega)$ and that $f_{n}(x) \to g(x)$ for a.e.\ $x \in \Omega$. Assume that $f \neq g$ on a set of finite positive measure $\Omega_{0} \subset \Omega$ on which $g$ is bounded. By Egorov's theorem there exists a set of positive measure $\Omega_{1} \subset \Omega_{0}$ such that $f_{n} \to g$ uniformly in $\Omega_{1}$. Therefore, $f_{n} \to g$ in $\Lp[p](\Omega_{1})$ and hence $f_{n} \toweak g$ in $\Lp[p](\Omega_{1})$. This contradicts the uniqueness of weak limits and therefore implies $f=g$ throughout $\Omega$.}.

    We now prove uniqueness of weak solutions. To this end let $c_0^j\in Z_0$, $j=1,2$, and let $c_j(t)$ be weak solutions to~\eqref{eq:Abstract1} with initial values $c_j(0) = c_0^j$. Set $c \eqdef c_{1} - c_{2}$ and $\mu \eqdef \mu_{1} - \mu_{2}$. Then multiplying the equation
    \begin{equation*}
        \dt{c} = \lapl_N \mu
    \end{equation*}
    by $c(t)$ in $\Hzero(\Omega)$  we deduce
    \begin{equation*}
        \frac{1}{2} \timeder \norm{\Hs[-1]_{(0)}}{c(t)}^{2} + \mathcal{E}(c(t), c(t)) + \Ltwoprod{\phi'(c(t))}{c(t)} = d\Lpnorm{c(t)}^{2}.
    \end{equation*}
    Using the inequality
    $$
    \Lpnorm{w}^{2} \leqslant \norm{\Hs[\sfrac{\alpha}{2}]_{(0)}}{w}^{\sfrac{4}{(2 + \alpha)}} \norm{\Hs[-1]_{(0)}}{w}^{\sfrac{2\alpha}{(2 + \alpha)}} \leqslant \Const_{\epsilon} \norm{\Hs[-1]_{(0)}}{w}^{2} + \epsilon \norm{\Hs[\sfrac{\alpha}{2}]_{(0)}}{w}^{2}
    $$
    and the coercivity of $\mathcal{E}$, that is,
    \begin{equation*}
        \frac{1}{\Const} \norm{\Hs[\sfrac{\alpha}{2}]_{(0)}}{c(t)}^{2} \leqslant \mathcal{E}(c(t), c(t)) + \Ltwoprod{\phi'(c(t))}{c(t)},
    \end{equation*}
    we infer
    \begin{equation*}
        \frac{1}{2} \timeder \norm{\Hs[-1]_{(0)}}{c(t)}^{2} + \frac{1}{\Const} \norm{\Hs[\sfrac{\alpha}{2}]}{c(t)}^{2} \leqslant \Const \norm{\Hs[-1]_{(0)}}{c(t)}^{2}.
    \end{equation*}
    Hence Gronwall's lemma implies
    \begin{equation}\label{eq:continuous_dependence}
        \norm{\Hs[-1]_{(0)}}{c(t)}^{2} \leqslant e^{2\Const t} \norm{\Hs[-1]_{(0)}}{c_{0}}^{2},
    \end{equation}
    which entails uniqueness whenever $c_{0}^{1} = c_{0}^{2}$.

    Let us now prove the energy identity~\eqref{eq:energy_identity}. First we observe that by taking $\eta = (-\lapl_{N})^{-1}c$ and $\varphi = c$ in~\eqref{eq:PDE_variational1} and~\eqref{eq:PDE_variational2} respectively we deduce~\eqref{eq:energy_identity} for almost any $t>0$ in the same way as in \cite[Proof of Theorem 1.2]{AsymptoticCH}. In order to obtain the energy identity for all times, we observe that any weak solution can be approximated by a family of functions $\{ c_{\theta} \}_{\theta > 0}$ defined by~\eqref{eq:CHr}. From the regularity of the solution $c$, we know that $c \in \Cont([0,T]; \Hs[\beta](\Omega))$ for all $\beta < \frac{\alpha}{2}$, Moreover, for all positive times $t$, we have $c_{\theta} \in \Cont([0,T]; \Hs[\sfrac{\alpha}{2}](\Omega))$ and therefore $c_{\theta}(t) \in \Hs[\sfrac{\alpha}{2}](\Omega)$ with uniform bound in $\theta$. Passing to the limit for $\theta \to 0$ and arguing by contradiction, we deduce $c(t) \in \Hs[\sfrac{\alpha}{2}](\Omega)$ for all positive times. Finally, since solution departing from $c(t) \in \Hs[\sfrac{\alpha}{2}](\Omega)$ are unique and since we already know that the energy identity holds for almost all times, for any $t>0$ we can find $\overline{t} > t$ such that the two following identities hold
    \begin{align*}
        E(c(\overline{t})) &+ \int_{0}^{\overline{t}} \Lpnorm{\nabla \mu(s)}^{2} \, \mathrm{d}s = E(c_{0})\\
        E(c(\overline{t})) &+ \int_{t}^{\overline{t}} \Lpnorm{\nabla \mu(s)}^{2} \, \mathrm{d}s = E(c(t)).
    \end{align*}
    Taking the difference we deduce for any time $t$
    \begin{equation*}
        E(c(t)) + \int_{0}^{t} \Lpnorm{\nabla \mu(s)}^{2} \, \mathrm{d}s = E(c_{0}),
    \end{equation*}
    which is the desired energy identity for all $t>0$.

    We still have to prove the continuity of the map $c_0 \mapsto c(t)$. Observe that the strong continuity in $\Hs[-1]_{(0)}(\Omega)$ is an immediate consequence of the continuous dependence estimate~\eqref{eq:continuous_dependence}. Moreover, since
    \begin{equation*}
        Z_0 \ni c_0 \mapsto c(t) \in \Hs[\alpha]_{\loc}(\Omega) \cap \Hs[\sfrac{\alpha}{2}]_{(0)}(\Omega)
    \end{equation*}
    is a bounded mapping, interpolation yields the continuity $c_0 \mapsto c(t)$ with respect to the $\Hs[\gamma]_{(0)}(\Omega)$-norm with $\gamma< \tfrac{\alpha}{2}$. Because of the boundedness of $[0,\infty)\ni t\mapsto c(t)\in \Hs[\alpha]_{(0)}(\Omega)$, this mapping is also weakly continuous.

    Finally, note that the energy equality holding for all $t>0$ entails the continuity of solutions $c(t)$ with values in $\Hs[\sfrac{\alpha}{2}]_{(0)}(\Omega)$ so that $c(t) \to_{t \to 0} c_{0}$ in $\Hs[\sfrac{\alpha}{2}]_{(0)}(\Omega)$ using that weak convergence plus convergence of norms imply strong convergence. This finishes the proof.

\section{Long-time behavior}\label{S:attractor}

Here we describe the global asymptotic behavior of the dynamical system associated with~\eqref{eq:PDE_variational1}--\eqref{eq:PDE_variational2}. As above we can reduce to the case that $c_0$ has mean value zero by adding a suitable constant. Let us define the (metric) phase-space
$$
\mathcal{X} = \left\{ z\in\Hs[\sfrac{\alpha}{2}]_{(0)}(\Omega)\,:\, \int_{\Omega} f(z)\, \mathrm{d}x < \infty \right\}
$$
endowed with the metric
\begin{equation*}
         \quad d_{\mathcal{X}}(z_{1}, z_{2}) = \norm{\Hs[\sfrac{\alpha}{2}]_{(0)}(\Omega)}{z_{1} - z_{2}} + \left| \int_{\Omega} f(z_{1}) \, \mathrm{d}x - \int_{\Omega} f(z_{2}) \, \mathrm{d}x \right|.
\end{equation*}
Thanks to Theorem~\ref{thm:Existence} and inequality~\eqref{eq:continuous_dependence}, we can define a closed semigroup (see \cite{PZ-07}) on $\mathcal{X}$ by setting $S(t)c_0=c(t)$, where $c$ is the unique weak solution to~\eqref{eq:PDE_variational1}--\eqref{eq:PDE_variational2} with initial datum $c_0$.

Our result is the following
\begin{theorem}\label{thm:attractor_existence}
   The dynamical system $(\mathcal{X},S(t))$ has a (connected) global attractor.
\end{theorem}

\begin{proof} Let us show first that the dynamical system has a bounded absorbing set.
Consider equation~\eqref{eq:PDE_variational2} defining the chemical potential and choose $c$ as test function. From this we deduce that
    \begin{align}
        \mathcal{E}(c(t), c(t)) + \Ltwoprod{f'(c(t))}{c(t)} =& \Ltwoprod{\mu(t)}{c(t)} = \Ltwoprod{\mu(t)- \mean{\mu(t)}}{c(t)} \notag\\
        \leqslant& \Const \int_{\Omega} |\nabla \mu(x,t)| \, \mathrm{d}x \leqslant \frac{1}{2} \Lpnorm{\nabla \mu}^{2} + \Const \label{eq:complete_energy_1}
    \end{align}
    holds for almost every $t \geqslant 0$. Here we used the fact that $c(t)$ has zero mean and that it is pointwise bounded. Moreover, from the assumptions on the potential~$f$ we have
    \begin{equation*}
        f(c) = \phi(c) - \frac{d}{2} c^{2}
    \end{equation*}
    where $\phi$ is convex. By the convexity of~$\phi$ we deduce
    \begin{equation*}
        \phi'(s)s \geqslant \phi(s) - \phi(0) \qquad \forall s
    \end{equation*}
    and therefore we can write
    \begin{align*}
        & \Ltwoprod{f'(c(t))}{c(t)} = \Ltwoprod{\phi'(c)}{c} - d\Lpnorm{c(t)}^{2}\\
        \geqslant& \int_{\Omega} \phi(c(t))\, \mathrm{d} x - |\Omega| \phi(0) - d\Lpnorm{c(t)}^{2} = \int_{\Omega} f(c(t))\,\mathrm{d} x - \Const.
    \end{align*}
    Substituting this estimate from below in the inequality~\eqref{eq:complete_energy_1} above we get
    \begin{equation*}
        \mathcal{E}(c(t), c(t)) + \int_{\Omega} f(c(t))\,\mathrm{d} x \leqslant \frac{1}{2} \Lpnorm{\nabla \mu(t)}^{2} + \Const.
    \end{equation*}
    We now consider the energy identity~\eqref{eq:energy_identity} and differentiate it with respect to time. This gives
    \begin{equation*}
        \timeder \left( \mathcal{E}(c(t), c(t)) + \int_{\Omega} f(c(t))\,\mathrm{d} x \right) + \Lpnorm{\nabla \mu(t)}^{2} \leqslant 0.
    \end{equation*}
    Summing the last two inequalities together, we infer
    \begin{equation*}
        \timeder \left( \mathcal{E}(c(t), c(t)) + \int_{\Omega} f(c(t))\,\mathrm{d} x \right) + \mathcal{E}(c(t), c(t)) + \int_{\Omega} f(c(t))\,\mathrm{d} x \leqslant \Const
    \end{equation*}
    for almost every $t \geqslant 0$. Gronwall's Lemma thus gives
    \begin{equation*}
        \mathcal{E}(c(t), c(t)) + \int_{\Omega} f(c(t))\,\mathrm{d} x \leqslant e^{-t} \left( \mathcal{E}(c_{0}, c_{0}) + \int_{\Omega} f(c_{0})\,\mathrm{d} x \right) + \Const
    \end{equation*}
    where the constant~$\Const$ appearing on the right hand side is independent of the initial datum~$c_{0}$. This proves that there is a bounded absorbing set $\mathcal{B}$ in $\mathcal{X}$.

    On account of \cite[Thm.~2]{PZ-07}, we only need to prove that there exists a divergent sequence $\{t_n\}$ such that $\alpha(S(t_n)B)=0$ as $n$ goes to $\infty$.
    Here $\alpha(E)$ denotes the Kuratowski measure of noncompactness. Actually we prove more, that is, $\alpha(S(t)B)=0$ for all $t>0$.

     Let $\{c_{0n}\}\subset \mathcal{X}$ be bounded and set $c_{n}(t)=S(t)c_{0n}$. From estimates analogous to those deduced in our proof of existence of solutions in Section~\ref{S:existence} and thanks to the existence of the absorbing set deduced above, we have the following (uniform with respect to $n$) estimates:
    \begin{alignat*}{2}
        c_{n}       &\in \Bochner{\Lp[\infty]}{0}{T}{\Hs[\sfrac{\alpha}{2}]_{(0)}(\Omega)}      &\qquad &\text{for all $T > 0$}\\
        \dt{c_{n}}  &\in \Bochner{\Lp}{s}{T+s}{\Hs[\sfrac{\alpha}{2}]_{(0)}(\Omega)}            &\qquad &\text{uniformly in $s \geqslant\varepsilon>0$, for all $T > 0$}\\
        \mu_{n}     &\in \Bochner{\Lp}{s}{T+s}{\Hs(\Omega)}                                     &\qquad &\text{uniformly in $s \geqslant 0$, for all $T > 0$}\\
        f'(c_{n})   &\in \Bochner{\Lp[\infty]}{\epsilon}{T}{\Lp(\Omega)}                        &\qquad &\text{for all $T > \epsilon > 0$}.
    \end{alignat*}
    Arguing as in the proof of existence, up to a subsequence, we deduce that there exist functions $c, \mu$ satisfying for any fixed $\epsilon$ and $T$
    \begin{alignat*}{2}
        c_{n}       &\toweakstar c      &\quad &\text{in $\Bochner{\Lp[\infty]}{0}{T}{\Hs[\sfrac{\alpha}{2}]_{(0)}(\Omega)}$}\\
        \dt{c_{n}}  &\toweak \dt{c}     &\quad &\text{in $\Bochner{\Lp}{s}{T+s}{\Hs[\sfrac{\alpha}{2}]_{(0)}(\Omega)}$}\\
        \mu_{n}     &\toweak \mu        &\quad &\text{in $\Bochner{\Lp}{s}{T+s}{\Hs(\Omega)}$}\\
        f'(c_{n})   &\toweakstar f'(c)  &\quad &\text{in $\Bochner{\Lp[\infty]}{\epsilon}{T}{\Lp(\Omega)}$}
    \end{alignat*}
    as $n\to\infty$. Here $c$ and $\mu$ satisfy \eqref{eq:PDE_variational1}--\eqref{eq:PDE_variational2}. On the other hand we also know that
    \begin{equation*}
        c_{n} \in \Cont( [0,T]; \Hs[\sfrac{\alpha}{2}]_{(0)}(\Omega) ) \qquad \text{for all $T > 0$}
    \end{equation*}
    with a uniform bound in $n$. Moreover, the estimate on $\dt{c_{n}}$ implies that the family $\{ c_{n} \}_{n \in \mathbb{N}}$ is also equicontinuous with values in $\Hs[\sfrac{\alpha}{2}]_{(0)}(\Omega)$. Indeed, this follows from the following simple computation
    \begin{align*}
        &\norm{\Hs[\sfrac{\alpha}{2}]_{(0)}(\Omega)}{c_{n}(t) - c_{n}(s)}\\
        \leqslant& \int_{s}^{t} \norm{\Hs[\sfrac{\alpha}{2}]_{(0)}(\Omega)}{\dt{c_{n}}(\tau)} \, \mathrm{d}\tau \leqslant (t-s)^{\sfrac{1}{2}} \left( \int_{s}^{t} \norm{\Hs[\sfrac{\alpha}{2}]_{(0)}}{\dt{c_{n}}(\tau)}^{2} \, \mathrm{d}\tau \right)^{\sfrac{1}{2}}.
    \end{align*}
    Moreover, we know that
    \begin{equation*}
        \mathcal{L} c(t) = \mu(t) - f'(c(t)) \in \Lp_{(0)}(\Omega)
    \end{equation*}
    holds for almost any $t > 0$. In particular, in the last expression, the right hand side is uniformly bounded for all $t \geqslant \epsilon > 0$. Since $\Lp_{(0)}(\Omega) \compactsubset \Hs[-\sfrac{\alpha}{2}]_{(0)}(\Omega)$ and since $\mathcal{L}^{-1}$ is continuous as an operator from $\Hs[-\sfrac{\alpha}{2}]_{(0)}(\Omega)$ to $\Hs[\sfrac{\alpha}{2}]_{(0)}(\Omega)$ we deduce that
    \begin{equation*}
        c(t) \in K_{\epsilon} \compactsubset \Hs[\sfrac{\alpha}{2}]_{(0)}(\Omega) \qquad \text{for a.e.\ $t \geqslant \epsilon > 0$}
    \end{equation*}
    holds with $K_{\epsilon}$ independent of $t$. Moreover, thanks to the continuity of the solutions $c_{n}$ taking values in $\Hs[\sfrac{\alpha}{2}]_{(0)}(\Omega)$ this is true for all $t \geqslant \epsilon > 0$. Appealing to the Ascoli-Arzel\`a Theorem (see, for instance, \cite[Theorem~4.43]{Folland1999}), this implies that, up to a subsequence, $c_{n} \to c$ in $\Cont([\epsilon, T]; \Hs[\sfrac{\alpha}{2}]_{(0)}(\Omega))$ uniformly on the interval $[\epsilon, T]$. Applying the above argument to intervals of the form $[m^{-1}, m]$, $m \in \mathbb{N}$ and performing a diagonal selection procedure, we finally obtain
    \begin{equation*}
        c_{n}(t) \to c(t) \qquad \text{in $\Hs[\sfrac{\alpha}{2}]_{(0)}(\Omega)$, for all $t > 0$}.
    \end{equation*}
    We now consider the convergence
    \begin{equation*}
        \int_{\Omega} f(c_{n}(t)) \, \mathrm{d}x \to \int_{\Omega} f(c(t)) \, \mathrm{d}x \qquad \text{for all $t > 0$.}
    \end{equation*}
    This follows as a consequence of the dominated convergence on account of the boundedness of $f$ and pointwise convergence almost everywhere in $\Omega$ of $\{f(c_{n}(\cdot, t))\}$. The latter is implied by the strong convergence of $c_{n}(t)$ in $\Hs[\sfrac{\alpha}{2}]_{(0)}(\Omega)$.
    Summing up we have that, up to a subsequence, $c_n(t)$ converges to $c(t)$ in $\mathcal{X}$ for any $t>0$. Therefore we are in a position to apply \cite[Thm.~2 and Prop.~4]{PZ-07} to conclude the proof.
\end{proof}

\begin{remark}
    In order to prove the connectedness of the attractor, in~\cite{PZ-07} (see comments after Theorem~2) it is assumed that balls of the phase space $\mathcal{X}$ are connected. However, connectedness of the attractor can be obtained by assuming only that the whole phase space is connected as was shown in~\cite[Thm.~4.2 and Cor.~.3]{Ball1997} (cf. also \cite{Ball1998}). In our case, although connectedness of the balls of $\mathcal{X}$ does not seem evident, nonetheless the $d$-convexity of $F$ implies that $\mathcal{X}$ is connected.
\end{remark}

\section{Boundary conditions for variational solutions}\label{S:BC_regular}

In this section we want to discuss the natural boundary condition satisfied by the weak solution $u$ to the problem
\begin{equation}\label{eq:nonlocal_elliptic}
    \mathcal{E}(u,\psi) = \Ltwoprod{g}{\psi} \qquad \forall \psi \in \Hs[\sfrac{\alpha}{2}]_{(0)}(\Omega).
\end{equation}
Here $g$ is a given function with $m(g)=0$. Of course, we can confine ourselves to consider the linear nonlocal equation neglecting the derivative of the potential $f$. Note that \eqref{eq:nonlocal_elliptic} also holds true for all $\psi\in \Hs[\sfrac{\alpha}{2}](\Omega)$ since both sides vanish on constants. For simplicity we only consider the case $\Omega=\R^n_+$. But the case of a bounded sufficiently smooth domain can be reduced to this case by standard techniques.

The main theorem is a conditional result, namely,
\begin{theorem}\label{thm:boundary_regularity}
 {Let $\Omega=\R^n_+$ and $1\leq p\leq \infty$ such that $\alpha-1>\frac{n}p$.}    Let $\alpha > \frac{3}{2}$ and let $u \in \Hs[\sfrac{\alpha}{2}]_{(0)}(\Omega) \cap \Hs[\alpha]_{\loc}(\Omega)$ be a solution to~\eqref{eq:nonlocal_elliptic} with $g\in L^p_{(0)}(\Omega)$. Suppose that $u \in \Cont[1,\beta](\closure{\Omega})$. If $n\geqslant 2$, we assume that the following limit exists
 {   \begin{equation}\label{eq:direction_kernel}
        \vect{n}_{x_{0}} = \lim_{\delta \to 0} \delta^{-1-n+\alpha}{\int_{\Omega} \int_{ \Omega} (x-y) (\varphi_\delta(x) - \varphi_\delta(y)) k(x, y, x-y) \, \mathrm{d}x \mathrm{d}y}
    \end{equation}
  and is non-zero,  where
  \begin{equation*}
        \varphi_{\delta}(x) =
        \begin{cases}
            1 - \frac{| x - x_{0} |}{\delta}    &\text{if $|x-x_{0}| < \delta$}\\
            0                               &\text{otherwise.}
        \end{cases}
      \end{equation*}
    }
    If $n=1$, let $ \vect{n}_{x_{0}}=1$. Then we have
    \begin{equation}\label{eq:Neumannbc}
      \nabla u(x_{0}) \cdot \vect{n}_{x_{0}} = 0 \qquad \forall\,x_0\in\partial \Omega.
    \end{equation}
  \end{theorem}

\begin{remark}\label{rem:equation a.e.}
    Observe that~\eqref{eq:nonlocal_elliptic} holds in particular for any $\psi \in \Cont[\infty]_{0}(\Omega)$. Arguing as in~\cite[Lemma~4.2]{Abels2007} (see also~\cite[Lemma~3.5]{Abels2007}), we have
    \begin{align*}
        &\mathcal{E}(u, \psi)\\
        =& \lim_{\epsilon \to 0} \frac{1}{2} \int_{ \{ |x-y| > \epsilon \} \cap \Omega} (u(x) - u(y))(\psi(x) - \psi(y)) k(x, y, x-y) \, \mathrm{d}x\mathrm{d}y\\
        =& \lim_{\epsilon \to 0} \frac{1}{2} \int_{\Omega} \int_{ \Omega \setminus B_{\epsilon}(x) } (u(x) - u(y)) \psi(x) k(x, y, x-y) \, \mathrm{d}x\mathrm{d}y\\
        & {} + \lim_{\epsilon \to 0} \frac{1}{2} \int_{\Omega} \int_{ \Omega \setminus B_{\epsilon}(x) } (u(x) - u(y)) \psi(x) k(y, x, y-x) \, \mathrm{d}x\mathrm{d}y\\
        =& \lim_{\epsilon \to 0} \int_{\Omega} \int_{ \Omega \setminus B_{\epsilon}(x) } (u(x) - u(y)) k(x, y, x-y) \, \mathrm{d}y \, \psi(x) \mathrm{d}x\\
        =& \int_{\Omega} \mathcal{L}u(x) \psi(x) \, \mathrm{d}x
    \end{align*}
    provided that $u \in \Hs[\alpha]_{\loc}(\Omega)\cap \Hs[\sfrac{\alpha}2](\Omega)$. Therefore,
    \begin{equation*}
        \Ltwoprod{\mathcal{L}u}{\psi} = \Ltwoprod{g}{\psi}
    \end{equation*}
    holds for all $\psi \in \Cont[\infty]_{0}(\Omega)$. This implies that the weak solution $u$ to~\eqref{eq:nonlocal_elliptic} satisfies the equation
    \begin{equation*}
        \mathcal{L}u(x) = g(x) \qquad \text{a.e.\ $x \in \Omega$.}
    \end{equation*}
    Hence $\mathcal{L}u\in L^p(\Omega)$ since $g\in L^p(\Omega)$. Unfortunately, we cannot say anything on the boundary condition due to the lack of information about the regularity of $u$. More precisely, we cannot answer to the question: Does $u$ belong to $\Hs[\alpha]_{(0)}(\Omega)$ with $\alpha > \frac{3}{2}$. In addition, assuming that the limit \eqref{eq:direction_kernel} exists, we do not know if
    \begin{equation*}
        \nabla u(x_{0}) \cdot\vect{n}_{x_0}=0
    \end{equation*}
    holds on $\partial \Omega$ in the sense of traces under general assumptions.
\end{remark}
We can thus just conjecture that $u$ should satisfy~\eqref{eq:Neumannbc}.

\begin{remark}\label{rem:HomKernel}
    The limit~\eqref{eq:direction_kernel} exists in many examples of interaction kernel which are interesting for applications. Among them there are kernels given by a homogeneous principal part of order $\alpha$ perturbed by lower order terms. For instance
    \begin{equation}\label{eq:IsoKernel}
        k(x, y, x-y) = \frac{\Const}{|x-y|^{n + \alpha}} + o (|x-y|^{-n-\alpha})
    \end{equation}
    or, more generally,
    \begin{equation*}
        k(x, y, x-y) = \frac{g(x,y)}{|x-y|^{n + \alpha}} + o (|x-y|^{-n-\alpha})
    \end{equation*}
    with $g\in \Cont(\overline{\Omega}\times \overline{\Omega} )$. In the latter cases a simple calculation using the homogeneity of $|x-y|^{-n-\alpha}$ and the continuity of $g$ in $(x_0,x_0)$ yields
{    \begin{equation*}
        \vect{n}_{x_{0}} = \int_{\Rn_+} \int_{\Rn_+} (x-y) (\varphi_1(x)-\varphi_1(y)) k(x_0, x_0, x-y) \, \mathrm{d}x \mathrm{d}y.
    \end{equation*}}
\end{remark}

\begin{remark}\label{rem:IsoKernel}
    In the case \eqref{eq:IsoKernel}, using the higher-order term symmetry, we have
{    \begin{equation*}
        \vect{n}_{x_{0}} =\int_{\Rn_+} \int_{\Rn_+} (x-y) (\varphi_1(x)-\varphi_1(y)) k(x_0, x_0, x-y) \, \mathrm{d}x \mathrm{d}y. = \Const \vect{\nu}
    \end{equation*}
    where $\vect{\nu}$ is the unit outward normal to the boundary and $C\neq 0$.} Therefore, in this case we recover the usual Neumann boundary conditions.
\end{remark}

\subsection{Proof of Theorem~\ref{thm:boundary_regularity}: case $n=1$}

{Before proving Theorem~\ref{thm:boundary_regularity} in the case $n\geq 2$, we first discuss the simpler one-dimensional case. For $\mathbb{R}^{n}_+$ with $n \geqslant 2$ the same general strategy applies with the required changes (see Section~\ref{S:boundary_general}). Througout this section we assume that $u$ is as in the assumption of Theorem~\ref{thm:boundary_regularity} and $n=1$. }

Consider the cut-off function at $x=0$ defined by
\begin{equation*}
    \varphi_{\delta}(x) =
    \begin{cases}
        1 - \frac{x}{\delta}    &\text{for $x \in [0, \delta)$,}\\
        0                       &\text{otherwise.}
    \end{cases}
\end{equation*}
Observe also that the function $\varphi_{\delta}$ is Lipschitz continuous and hence belongs to $\Hs[\gamma](\Omega)$ for any $\gamma \in [0,1]$ and thus to $\Hs[\sfrac{\alpha}{2}](\Omega)$.

Using \eqref{eq:nonlocal_elliptic}, we obtain
 \begin{align*}
   \left|\mathcal{E}(u,\varphi_\delta)\right| &= \left| \int_{0}^{\infty} g(x) \varphi_{\delta}(x) \, \mathrm{d}x \right|\\
 &\leqslant \|g\|_p  \|\varphi_{\delta}\|_{p'} \leqslant C \|g\|_{p} {\delta^{\sfrac1{p'}}}= O(\delta^{\sfrac{1}{p'}}) \qquad \text{as $\delta \to 0$.}
 \end{align*}
We now consider the quantity $\mathcal{E}(u, \varphi_{\delta})$ in more detail. We have
\begin{align}\label{eq:boundary_1D_bilinear}
    &\mathcal{E}(u, \varphi_{\delta}) \notag \\
    =& \frac{1}{2} \int_{0}^{\infty} \int_{0}^{\infty} (u(x) - u(y)) (\varphi_{\delta}(x) - \varphi_{\delta}(y)) k(x, y, x-y) \, \mathrm{d}x\mathrm{d}y \notag\\
    =& \frac{1}{2} \int_{0}^{\delta} \int_{0}^{\delta} (u(x) - u(y)) \left( \left( 1 - \frac{x}{\delta} \right)  - \left( 1 - \frac{y}{\delta} \right)  \right) k(x, y, x-y) \, \mathrm{d}x\mathrm{d}y \notag\\
    & {} + \frac{1}{2} \int_{\delta}^{\infty} \int_{0}^{\delta} (u(x) - u(y)) \left( 1 - \frac{x}{\delta} \right)  k(x, y, x-y) \mathrm{d}x \mathrm{d}y \notag\\
    & {} - \frac{1}{2} \int_{0}^{\delta} \int_{\delta}^{\infty} (u(x) - u(y)) \left( 1 - \frac{y}{\delta} \right)  k(x, y, x-y) \mathrm{d}x
    \mathrm{d}y\notag\\
    =& I_1 + I_2 + I_3.
\end{align}
The first integral reduces to
\begin{equation*}
    I_{1} = \frac{1}{2 \delta} \int_{0}^{\delta} \int_{0}^{\delta} (u(x) - u(y))(y-x) k(x, y, x-y) \, \mathrm{d}x\mathrm{d}y.
\end{equation*}
By using Taylor series expansion near $0$ for $u$, this integral can be estimated by
\begin{align*}
    I_{1} \sim & \frac{1}{2 \delta} \int_{0}^{\delta} \int_{0}^{\delta} u'(0) (x-y) (y-x) k(x, y, x-y) \, \mathrm{d}x\mathrm{d}y\\
    &\quad {} + \frac{\Const}{\delta} \int_{0}^{\delta} \int_{0}^{\delta} (x-y)^{1+\beta} (y-x) k(x, y, x-y) \, \mathrm{d}x\mathrm{d}y\\
    =& - \frac{1}{2\delta} u'(0) \int_{0}^{\delta} \int_{0}^{\delta} (x-y)^{2} k(x, y, x-y) \mathrm{d}x\mathrm{d}y + O(\delta^{2+\beta-\alpha}),
\end{align*}
Here the notation $A \sim B$ means that the quantities $A$ and $B$ are equivalent for $\delta \to 0$, i.e., that, for $\delta$ sufficiently small, there exist two positive constants $\const$ and $\Const$ such that $c B \leqslant A \leqslant \Const B$. Indeed, notice that by assumption~\eqref{k-ass-three}, $(x-y)^{1+\alpha} k(x, y, x-y)$ is uniformly bounded away from zero from below and from above.

In the sequel, $\omega(\delta^{\gamma})$ indicates a quantity which is asymptotical to $\delta^{\gamma}$ in the following sense
\begin{equation*}
    g(\delta) \in \omega(\delta^{\gamma}) \qquad \text{if} \qquad \const \leqslant \liminf_{\delta \to 0} \frac{g(\delta)}{\delta^{\gamma}} \leqslant \limsup_{\delta \to 0} \frac{g(\delta)}{\delta^{\gamma}} \leqslant \Const, \qquad \const, \Const > 0.
\end{equation*}
From the above computations we deduce
\begin{equation*}
    I_{1} = -u'(0) \omega(\delta^{2-\alpha}) + O(\delta^{2 + \beta -\alpha}).
\end{equation*}

The remaining two terms $I_2$ and $I_3$ in~\eqref{eq:boundary_1D_bilinear} are equivalent. Thus it suffices to observe that
\begin{align*}
    I_{3} =& \int_{0}^{\delta} \int_{\delta}^{\infty} (u(y) - u(x) + u'(y)(x-y)) \left( 1 - \frac{y}{\delta} \right) k(x, y, x-y) \, \mathrm{d}x \mathrm{d}y\\
    & \quad {} + \int_{0}^{\delta} \left( \int_{\delta}^{\infty} (y-x) k(x, y, x-y) \, \mathrm{d}x \right) u'(y) \left( 1 - \frac{y}{\delta} \right) \mathrm{d}y.
\end{align*}
However, the first of these two terms are of order $O(\delta^{2+\beta-\alpha})$, while we can easily estimate the inner integral appearing in the second one as
\begin{align*}
    &\int_{\delta}^{\infty} (y-x) k(x, y, x-y) \, \mathrm{d}x \sim \int_{\delta}^{\infty} (y-x) |x-y|^{-1-\alpha} \, \mathrm{d}x \\
    &\sim - \int_{\delta}^{\infty} (x-y)^{-\alpha} \, \mathrm{d}x \sim - (\delta - y)^{1 - \alpha},
\end{align*}
where we used the fact that $x > y$ always when $I_{3}$ will be computed. Therefore we deduce
\begin{align*}
    I_{3} \sim& - \int_{0}^{\delta} (\delta - y)^{1 - \alpha} u'(y) \left( 1 - \frac{y}{\delta} \right) \mathrm{d}y + O(\delta^{2 + \beta - \alpha})\\
    \sim& - \frac{1}{\delta} \int_{0}^{\delta} (\delta - y)^{2 - \alpha} u'(0) \, \mathrm{d}y + O(\delta^{2+\beta-\alpha})\\
    \sim& \frac{1}{\delta} u'(0) \left. \frac{(\delta-y)^{3 - \alpha}}{3 - \alpha} \right|_{0}^{\delta} + O(\delta^{2 + \beta - \alpha}) \sim - u'(0) \delta^{2 - \alpha} + O(\delta^{2 + \beta - \alpha}).
\end{align*}
Combining all the above estimates we obtain, for some positive constants $\const$ and $\Const$,
\begin{equation*}
    - \Const u'(0) \delta^{2 - \alpha} - \Const \delta^{2+\beta-\alpha} \leqslant \mathcal{E}(u, \varphi_{\delta}) \leqslant - \const u'(0) \delta^{2-\alpha} + \Const \delta^{2 + \beta - \alpha}.
\end{equation*}
Because of \eqref{eq:nonlocal_elliptic}, we have
\begin{equation*}
    \mathcal{E}(u, \varphi_{\delta}) = \int_{0}^{\infty} g(x)\varphi_{\delta}(x)  \mathrm{d}x
\end{equation*}
so that
\begin{equation*}
    -u'(0) \delta^{2 - \alpha} + O(\delta^{2 + \beta - \alpha}) = O(\delta^{\sfrac{1}{p'}}) \qquad \text{as $\delta \to 0$.}
\end{equation*}
However, since $\alpha-1 > \tfrac{1}{p}$, one readily sees that this is possible only if  $u'(0) = 0$.

\subsection{Proof of Theorem~\ref{thm:boundary_regularity}: case $n\geqslant 2$}\label{S:boundary_general}

Here we consider the case $n\geqslant 2$ in the statement of Theorem~\ref{thm:boundary_regularity}. In this section, unless otherwise stated, we will denote by $B^{+}_{\delta}$ the half ball of center $x_{0} \in \partial \mathbb{R}^{n}_{+}$ contained in the half-space $\mathbb{R}^{n}_{+}$ and having radius $\delta$. Before giving the details of the proof we prove a useful technical lemma.
\begin{lemma}\label{L:integral}
    Let $r$ be a real number such that $r > -1-n$. Then the following relation holds
    \begin{equation*}
        I_{r} \eqdef \int_{B^{+}_{\delta}} \int_{B^{+}_{\delta}} |x-y|^{r} \left| |x{-x_0}| - |y{-x_0}| \right| \, \mathrm{d}x \mathrm{d}y = \omega(\delta^{1 + r +2n}) \qquad \text{for $\delta \to 0$.}
    \end{equation*}
\end{lemma}

\begin{proof}
    First of all by a simple translation we can reduce to the case $x_0=0$. Then using the change of variable $x=\delta \tilde{x}, y=\delta \tilde{y}$ we obtain
    \begin{equation*}
        I_r = \delta^{1+r+2n} \int_{B^{+}_{1}} \int_{B^{+}_{1}} |x-y|^{r} \left| |x| - |y| \right| \, \mathrm{d}x \mathrm{d}y,
    \end{equation*}
    where the last integral can be estimate as
    \begin{equation*}
        \int_{B^{+}_{1}} \int_{B^{+}_{1}} |x-y|^{r} \left| |x| - |y| \right| \, \mathrm{d}x \mathrm{d}y \leq\int_{B^{+}_{1}} \int_{B^{+}_{1}} |x-y|^{r+1} \, \mathrm{d}x \mathrm{d}y <\infty
    \end{equation*}
    since $r+1>-n$. Since the integral on the left-hand side is obviously positive, the statement follows.
\end{proof}

The main idea of our proof is to consider an appropriate family of test functions concentrating at one point of the boundary.  Let us consider the bilinear form
\begin{equation*}
    \mathcal{E}(u,v) = \frac{1}{2} \int_{\mathbb{R}^{n}_{+}} \int_{\mathbb{R}^{n}_{+}} (u(x) - u(y)) (v(x) - v(y)) k(x, y, x-y) \mathrm{d}x\mathrm{d}y.
\end{equation*}
Let $x_{0}$ be a point belonging to the boundary of $\mathbb{R}^{n}_{+}$ and consider the test function $\varphi_{\delta} \in \Cont(\mathbb{R}^{n}_{+}) \cap \Hs(\mathbb{R}^{n}_{+}) \subset \Hs[\sfrac{\alpha}{2}](\mathbb{R}^{n}_{+})$ defined by
\begin{equation*}
    \varphi_{\delta}(x) =
    \begin{cases}
        1 - \frac{x - x_{0}}{\delta}    &\text{if $|x-x_{0}| < \delta$ and $x \in \mathbb{R}^{n}_{+}$,}\\
        0                               &\text{otherwise.}
    \end{cases}
\end{equation*}
In particular, we have by \eqref{eq:nonlocal_elliptic}
\begin{equation*}
    \left|\mathcal{E}(u,\varphi_\delta)\right|=        \left| \int_{\mathbb{R}^{n}_{+}} g(x) \varphi_{\delta}(x) \, \mathrm{d}x \right| \leqslant{ \norm{L^p(\mathbb{R}^{n}_{+})}{g} \norm{L^{p'}(\mathbb{R}^{n}_{+})}{\varphi_{\delta}}.}
\end{equation*}
    Moreover, there holds
    \begin{equation*}
       \norm{L^p(\mathbb{R}^{n}_{+})}{\varphi_{\delta}} = \left(\int_{B_{\delta}^+} \left( 1 - \frac{|x - x_{0}|}{\delta} \right)^{p'} \, \mathrm{d}x\right)^{\sfrac1{p'}} \sim \delta^{\sfrac{n}{p'}}.
    \end{equation*}
Observe now that
\begin{align}
    \mathcal{E}(u, \varphi_{\delta})
    =& \frac{1}{2} \int_{\mathbb{R}^{n}_{+}} \int_{\mathbb{R}^{n}_{+}} (u(x) - u(y)) (\varphi_{\delta}(x) - \varphi_{\delta}(y)) k(x, y, x-y) \, \mathrm{d}x\mathrm{d}y \notag\\
    =& \frac{1}{2 \delta} \int_{B^{+}_{\delta}} \int_{B^{+}_{\delta}} (u(x) - u(y)) \left( |x - x_{0}| - |y - x_{0}| \right) k(x, y, x-y) \, \mathrm{d}x\mathrm{d}y \notag\\
    &\quad {} + \frac{1}{2} \int_{B^{+}_{> \delta}} \int_{B^{+}_{\delta}} (u(x) - u(y)) \left( 1 - \frac{|x-x_{0}|}{\delta} \right) k(x, y, x-y) \, \mathrm{d}x \mathrm{d}y \notag\\
    &\quad {} + \frac{1}{2} \int_{B^{+}_{\delta}} \int_{B^{+}_{> \delta}} (u(x) - u(y)) \left( 1 - \frac{|y-x_{0}|}{\delta} \right) k(x, y, x-y) \, \mathrm{d}x \mathrm{d}y \notag\\
    =& J_{1} + J_{2} + J_{3}. \label{eq:boundary_Rn_bilinear}
\end{align}
Here $B^{+}_{> \delta}$ denotes the set
\begin{equation*}
    B^{+}_{> \delta} \eqdef \{ x \in \mathbb{R}^{n}_{+} \mid |x-x_{0}| > \delta \}.
\end{equation*}
We must evaluate the asymptotic behavior of $J_1$, $J_2$ and $J_3$ as $\delta \to 0$.
Concerning the first one, using Taylor expansion, up to the first order we can write
\begin{equation*}
    u(x) - u(y) = \nabla u(x_{0}) \cdot (x-y) + {O(\delta^{\beta})|x-y|} \qquad \text{for $x, y \in B_{\delta}$}
\end{equation*}
where we have used
\begin{align*}
    &u(x) - u(y) - \nabla u(x_0) \cdot (x-y) \\
  =&\int_{0}^{1} (\nabla u(x + s(y-x))-\nabla u(x_0)) \cdot (x-y) \, \mathrm{d}s\\
  =&  O(\delta^{\beta}) |x-y|
\end{align*}
for all $x,y\in B_\delta$. Therefore, we obtain
\begin{align*}
    J_{1} &= \frac{1}{2 \delta} \nabla u(x_{0}) \cdot \int_{B^{+}_{\delta}} \int_{B^{+}_{\delta}} (x-y) \left( |x - x_{0}| - |y - x_{0}| \right) k(x, y, x-y) \, \mathrm{d}x\mathrm{d}y \\
    &+ O(\delta^{1 + \beta + n - \alpha}).
\end{align*}

We are now left to analyze $J_2$ and $J_3$ (see~\eqref{eq:boundary_Rn_bilinear}). This can be done similarly as in the case $n=1$. Indeed, we have
\begin{align*}
    J_{2} =& \frac{1}{2} \int_{B^{+}_{\delta}} \int_{B^{+}_{> \delta}} (u(x) - u(y)) \left( 1 - \frac{|x - x_{0}|}{\delta} \right) k(x, y, x-y) \, \mathrm{d}x \mathrm{d}y\\
    =& \frac{1}{2} \nabla u(x_{0}) \cdot \int_{B^{+}_{\delta}} \left( \int_{B^{+}_{> \delta}} (x-y) \left( 1 - \frac{|x - x_{0}|}{\delta} \right) k(x, y, x-y) \, \mathrm{d}x \right) \mathrm{d}y\\
    & \quad {} + O(\delta^{\beta}) \int_{B^{+}_{\delta}} \left( \int_{B^{+}_{> \delta}} |x-y| \left( 1 - \frac{|x - x_{0}|}{\delta} \right) k(x, y, x-y) \, \mathrm{d}x \right) \mathrm{d}y\\
    &\equiv J_{2,1}+J_{2,2}
\end{align*}
Thus, in this case, we only have to bound integrals of the form
\begin{align*}
    J_4 \eqdef &\int_{B^{+}_{\delta}} \left( \int_{B^{+}_{> \delta}} |x-y|^{1 - \alpha - n} \left( 1 - \frac{|x{-x_0}|}{\delta} \right) \mathrm{d}y \right) \mathrm{d}x\\
    =& \int_{B^{+}_{\delta}} \left( 1 - \frac{|x{-x_0}|}{\delta} \right) \left( \int_{B^{+}_{> \delta}} |x-y|^{1 - \alpha - n}  \, \mathrm{d}y \right) \mathrm{d}x.
\end{align*}
Let us set
\begin{equation*}
    J_{x} \eqdef \int_{B^{+}_{> \delta}} |x - y|^{1 - \alpha - n} \, \mathrm{d}y
\end{equation*}
and observe that
\begin{equation*}
    J_{x} \leqslant \int_{B_{> (\delta - |x|)}} |y|^{1 - \alpha - n} \, \mathrm{d}y = \int_{\delta - |x|}^{+\infty} \rho^{-\alpha} \, \mathrm{d}\rho \sim (\delta - |x|)^{1 - \alpha}.
\end{equation*}
Then we have
\begin{equation*}
    J_4 \leqslant \Const \frac{1}{\delta} \int_{B^{+}_{\delta}} (\delta - |x|)^{2 - \alpha} \, \mathrm{d}x = \Const \delta^{1 + n - \alpha} \int_{0}^{1} (1 - t)^{2 - \alpha} t^{n-1} \, \mathrm{d}t \sim \delta^{1 + n - \alpha}.
\end{equation*}
From such estimates we deduce
\begin{equation*}
    J_{2} = J_{2,1} + O(\delta^{1 + \beta + n - \alpha}).
\end{equation*}
The quantity $J_3$ is controlled in the same way, just interchanging the role of $x$ and $y$.

Recalling that
\begin{equation*}
    \mathcal{E}(u, \varphi_{\delta}) = \int_{\mathbb{R}^{n}_{+}} \mathcal{L}u(x) \varphi_{\delta}(x) \, \mathrm{d}x
\end{equation*}
and using the above estimates for $J_1$, $J_2$ and $J_3$, we get
\begin{align*}
    &\nabla u(x_{0}) \cdot \int_{\Rn_+} \int_{\Rn_+} (x - y) (\varphi_\delta(x) - \varphi_\delta(y)) k(x, y, x-y) \, \mathrm{d}x \mathrm{d}y \\
    &+ O(\delta^{1 + \beta + n - \alpha}) = O(\delta^{\sfrac{n}{p'}}).
\end{align*}
Since the double integral belongs to $\omega(\delta^{1+n - \alpha })$ by Lemma~\ref{L:integral}, on account of~\eqref{eq:direction_kernel}, {and $1+n-\alpha< \frac{n}{p'}$ (which is equivalent to $\alpha-1>\frac{n}{p'}$)}, we finally deduce
\begin{equation*}
    \nabla u(x_{0}) \cdot \vect{n}_{x_{0}} = 0. \qedhere
\end{equation*}

\bibliography{Bibliography,Bibliography2}
\bibliographystyle{plain}

\end{document}